\documentclass[reqno,twoside,a4paper,11pt]{amsart}

\usepackage{amsthm, amssymb, amsmath, amsfonts,eufrak,mathrsfs,dsfont}
\usepackage{url}
\usepackage{verbatim}
\usepackage{color}
\usepackage[utf8]{inputenc}
\usepackage[displaymath,mathlines]{lineno}
\IfFileExists{epsf.def}{\input epsf.def}{\usepackage{epsf}}

\DeclareMathOperator*{\osc}{osc}

\DeclareFontFamily{OT1}{eusb}{} \DeclareFontShape{OT1}{eusb}{m}{n} {<5> <6> <7> <8> <9> <10> <11> <12> <14.4> eusb10}{}
\DeclareMathAlphabet{\eusb}{OT1}{eusb}{m}{n}

\DeclareFontFamily{OT1}{eusm}{} \DeclareFontShape{OT1}{eusm}{m}{n} {<5> <6> <7> <8> <9> <10> <11> <12> <14.4> eusm10}{}
\DeclareMathAlphabet{\eusm}{OT1}{eusm}{m}{n}

\DeclareFontFamily{OT1}{eufm}{} \DeclareFontShape{OT1}{eufm}{m}{n} {<5> <6> <7> <8> <9> <10> <11> <12> <14.4> eufm10}{}
\DeclareMathAlphabet{\mathfrak}{OT1}{eufm}{m}{n}

\DeclareFontFamily{OT1}{fraktura}{}
\DeclareFontShape{OT1}{fraktura}{m}{n} {<5> <6> <7> <8> <9> <10> <11> <12> <13> <14.4> [1.1] eufm10}{}
\DeclareMathAlphabet{\fraktura}{OT1}{fraktura}{m}{n}

\DeclareFontFamily{OT1}{cmfi}{} \DeclareFontShape{OT1}{cmfi}{m}{n} {<5> <6> <7> <8> <9> <10> <11> <12> <13> <14.4> [0.9] cmfi10}{}
\DeclareMathAlphabet{\cmfi}{OT1}{cmfi}{b}{n}

\DeclareFontFamily{OT1}{cmss}{} \DeclareFontShape{OT1}{cmss}{m}{n} {<5> <6> <7> <8> <9> <10> <11> <12> <13> <14.4> cmss10}{}
\DeclareMathAlphabet{\cmss}{OT1}{cmss}{m}{n}

\DeclareMathAlphabet{\mathpzc}{OT1}{pzc}{m}{it}

\newtheoremstyle{thm}{1.8ex}{1.8ex}{\itshape\rmfamily}{} {\bfseries\rmfamily}{}{2ex}{}

\newtheoremstyle{def}{1.8ex}{1.8ex}{\slshape\rmfamily}{} {\bfseries\rmfamily}{}{2ex}{}

\newtheoremstyle{rem}{1.8ex}{1.8ex}{\rmfamily}{} {\bfseries\rmfamily}{}{2ex}{}

\newenvironment{proofsect}[1] {\vspace{0.2cm}\noindent{\rmfamily\itshape#1.}}{\qed\vspace{0.15cm}}
\theoremstyle{thm}

\newcommand\cour[1]{{\fontfamily{pcr}\selectfont #1}}

\theoremstyle{thm}
\newtheorem{theorem}{Theorem}[section]
\newtheorem{lemma}[theorem]{Lemma}
\newtheorem{proposition}[theorem]{Proposition}

\newtheorem*{Main Theorem}{Main Theorem.}
\newtheorem{corollary}[theorem]{Corollary}
\newtheorem*{special theorem}{Lindeberg-Feller Theorem for Martingales}

\newtheorem{assump}[theorem]{Assumption}

\theoremstyle{def}

\theoremstyle{rem}
\newtheorem{remark}[theorem]{Remark}
\newtheorem{remarks}[theorem]{Remarks}

\numberwithin{equation}{section}

\renewcommand{\section}{\secdef\sct\sect}
\newcommand{\sct}[2][default]{%
\refstepcounter{section}
\addcontentsline{toc}{section}{{\tocsection {}{\thesection}{\!\!\!\!#1\dotfill}}{}}
\vspace{0.7cm}
\centerline{\scshape\thesection.\ #1} \nopagebreak \vspace{0.2cm}}
\newcommand{\sect}[1]{%
\vspace{0.4cm} \centerline{\large\scshape\rmfamily #1}
\vspace{0.2cm}}

\renewcommand{\subsection}{\secdef\subsct\sbsect}
\newcommand{\subsct}[2][default]{\refstepcounter{subsection}
\addcontentsline{toc}{subsection}
{{\tocsection{\!\!}{\hspace{1.2em}\thesubsection}{\!\!\!\!#1\dotfill}}{}}
\nopagebreak\vspace{0.45\baselineskip} {\flushleft\bf
\thesubsection~\bf #1.~}
\\*[3mm]\noindent
\nopagebreak}
\newcommand{\sbsect}[1]{\vspace{0.1cm}\noindent
\textbf{#1.~}\vspace{0.1cm}}

\renewcommand{\subsubsection}{%
\secdef \subsubsect\sbsbsect}
\newcommand{\subsubsect}[2][default]{%
\refstepcounter{subsubsection} 
\addcontentsline{toc}{subsubsection}{{\tocsection{\!\!}
{\hspace{3.05em}\thesubsubsection}{\!\!\!\!#1\dotfill}}{}}
\nopagebreak
\vspace{0.15\baselineskip} \nopagebreak {\flushleft\rmfamily
\itshape\thesubsubsection
\ \rmfamily #1\/.}\ }
\newcommand{\sbsbsect}[1]{\vspace{0.1cm}\noindent
\rmfamily \itshape
\arabic{section}.\arabic{subsection}.\arabic{subsubsection} \
\sffamily #1\/.\ }

\renewcommand{\caption}[1]{%
\vglue0.5cm
\refstepcounter{figure}
\begin{minipage}{0.9\textwidth}\small {\sc Figure~\thefigure. }#1\end{minipage}}

\def\myffrac#1#2 in #3{\raise 2.6pt\hbox{$#3 #1$}\mkern-1.5mu\raise 0.8pt\hbox{$#3/$}\mkern-1.1mu\lower 1.5pt\hbox{$#3 #2$}}

\definecolor{lightgray}{gray}{0.5}

\newcommand{\twocite}[2]{\cite{#1}--\cite{#2}}

\newcommand{\N}        {\mathbb N}
\newcommand{\R}        {\mathbb R}
\newcommand{\RR}        {\cmss R}
\newcommand{\B}        {\mathbb B}
\newcommand{\PP}        {\mathbb P}
\newcommand{\p}        {\cmss P}

\newcommand{\Z}        {\mathbb Z}

\newcommand{\E}        {\mathbb E}

\newcommand{\CC}          {\mathcal{C}}
\newcommand{\scrC}      {\mathscr{C}} 
\newcommand{\scrH}      {\mathscr{H}} 
\newcommand{\scrE}      {\mathscr{E}} 
 
\newcommand{\EE}          {\mathcal{E}}
 
\newcommand{\GG}          {\mathcal{G}} 
\newcommand{\G}          {\mathbb{G}} 
\newcommand{\LL}         {\mathcal{L}}
\newcommand{\LLL}         {\mathbb{L}}
\newcommand{\MM}       {\mathcal{M}}
\newcommand{\eps}    {\varepsilon}

 \newcommand{\tworef}[2]{\ref{#1} -- \ref{#2}}

\newcommand{\treeeqref}[3]{\eqref{#1}--\eqref{#2}--\eqref{#3}}

\newcommand{\diam}      {\hbox{{\rm diam}}}

\newcommand{\1}{\mathds 1}
\newcommand{\deq}{\mathrel{\mathop:}=}
\newcommand{\twopoints}{\mathrel{\mathop:}}
\newcommand{\cc}{{\text{\rm c}}}

\title[Harnack inequalities and Local-CLT]{Harnack inequalities and local central limit theorem 
for the polynomial lower tail \\ Random Conductance Model}
\author[O.~Boukhadra T.~Kumagai and P.~Mathieu]
{O.~Boukhadra$^1$, T.~Kumagai$^2$ \and\ P.~Mathieu$^3$
}

\begin{document} 

\maketitle
\vspace{-5mm}
\centerline{\sc Dedicated to the memory of Professor Kiyosi It\^o.}
\vspace{0.3cm}
\centerline{$^1$D\'epartement de Math\'ematiques, Universit\'e de Constantine 1} 
\centerline{\small\cour{\url{boukhadra@umc.edu.dz}}}
\smallskip
\centerline{$^2$Research Institute for Mathematical Sciences, Kyoto University,}
 \centerline{Kyoto 606-8502, Japan}
\centerline{\small\cour{\url{kumagai@kurims.kyoto-u.ac.jp}}}
\smallskip
\centerline{$^3$Centre de Math\'ematiques et Informatique} 
\centerline{Aix Marseille Universit\'e, CNRS, Centrale Marseille, I2M, UMR 7373,}
\centerline{13453 Marseille France}
\centerline{\small\cour{\url{pierre.mathieu@univ-amu.fr}}}
\smallskip

\vspace{-2mm}
\begin{abstract} 
\it{We prove upper bounds on the transition probabilities of random walks with i.i.d. random conductances with a polynomial lower tail near $0$. 
We consider both constant and variable speed models. Our estimates are sharp. 
As a consequence, we derive local central limit theorems, parabolic Harnack inequalities and Gaussian bounds for the heat kernel. 
Some of the arguments are robust and applicable for random walks on  general graphs. Such results are stated under a general setting.
}
\newline 
\noindent
{\textbf{keywords}}~:
Markov chains, Random walk,  Random environments,
Random conductances,  Percolation.
\end{abstract}

\section{\textbf{Introduction and Results}}

The work presented below mainly concerns the 
Random Conductance Model (RCM) 
with i.i.d. conductances that have polynomial lower tails at zero. 
We shall obtain various heat kernel bounds, Harnack inequalities and a local central-limit theorem for such models under sharp conditions on the fatness of the tail 
of the conductances near $0$. Some of our arguments exploit specific features of the model - mainly some geometric information on the field of conductances and 
its spectral implications - while other arguments are general properties of random walks on graphs. 
In the rest of this introduction, 
we will separate results that are more robust from those that are specific to the RCM. 
The robust results will be discussed in the first subsection below, and 
results specific to the RCM 
and references to the existing literature 
will be given in the second subsection. 
Readers who are interested in RCM may start reading this paper from the  second subsection. 

Notation: We use $c$ or $C$ as generic positive constants.

\subsection{Part I: Framework and the results}
In this subsection, we give 
sufficient conditions for 
various heat kernel bounds, Harnack inequalities and a local central-limit theorem on a general graph. The results will be used in the next subsection for a concrete RCM.

Let $(G,\pi)$ be a weighted graph. 
That is, $G$ is a countable set and 
$\omega_{xy}=\omega_{yx}\ge 0$ for each $x,y\in G$. We write $x\sim y$ if and only if $\omega_{xy}>0$. 
We assume $(G,\pi)$ is connected and it has bounded degree (i.e. there exists $M>0$ such that $|\{y\in G: \omega_{xy}>0\}|\le M$ for each $x\in G$). 
For $x\ne y$, $\ell (x,y)=\{x_0, x_1,\cdots, x_m\}$ is called a path from from $x$ to $y$
if $x=x_0, x_1,\cdots, x_m=y$ and $x_i\sim x_{i+1}$ for $i=0,\cdots, m-1$. Write $|\ell (x,y)|=m$. 
Define the graph distance by
$d(x,y)=\min\{|\ell(x,y)|: \ell(x,y)\in {\mathcal P}(x,y)\}
$ where ${\mathcal P}(x,y)$ is the set of paths from $x$ to $y$. We define $d(x,x)=0$
for $x\in G$. 
Write $B(x,R) := \{x\in G~: d(x,y) < R\}$ and $\bar B(x,R) := \{x\in G: d(x,y) \le R\}$. 
For $A\subset G$, define $\pi(A)=\sum_{x\in A}\pi(x)$ where  $\pi(x)=\sum_{y\sim x}\omega_{xy}$,
and $\nu(A)=\sum_{x\in A}\nu_x$ where $\nu_x\equiv 1$. 

We will consider VSRW (variable speed random walk) and CSRW (constant speed random walk) that correspond to $(G,\pi)$. Both are continuous time Markov chains whose transition probability from $x$ to $y$ is given by $\omega_{xy}/\pi(x)$. The holding time at $x$ is exponentially distributed with mean $\pi(x)^{-1}$ for VSRW and with mean $1$ for CSRW.  
The corresponding discrete Laplace operator and heat kernel can be written as 
 \[{\mathcal L}_\theta f(x)=\frac 1{\theta_x}\sum_y(f(y)-f(x))\omega_{xy},
 ~~~p_t^{(\theta)}(x,y)=P^x(X^{(\theta)}_t=y)/\theta_y,\]
where $\theta_x = \theta (x) = \pi(x)$ for CSRW and $\theta_x = 1$ for VSRW. Thus the notation ${\mathcal L}_\pi$ and $X^{(\pi)}$ correspond to CSRW and ${\mathcal L}_\nu$, $X^{(\nu)}$ correspond to VSRW. We may and will often remove the script when results are valid for both types of random walks.

Let $\tilde d(\cdot,\cdot)$ be a metric defined by
\[
\tilde d (x,y)=\min\{\sum_{i=0}^{m-1} (1\wedge \omega_{x_ix_{i+1}}^{-1/2}):
\ell (x,y)=\{x_0, x_1,\cdots, x_m\}\in {\mathcal P}(x,y)\}. 
\]
Note that by definition, it is clear that $\tilde d(x,y)\le d(x,y)$ for all $x,y\in G$.
Write $\tilde B(x,R) := \{x\in\Z^d~: \tilde d(x,y) < R\}$. 
For $A\subset G$, let $\tau_A=\inf \{t\ge 0: X_t\notin A\}$.

\medskip 

In the following, we fix $\theta$ (which is either $\pi$ or $\nu$) and
consider either CSRW or VSRW. 
\begin{assump}\label{asmp:1-1asp} Let $x_0\in G$ be a distinguished point.\\
(i) There exist $\delta> 0, c_1>0$ and $T_0(x_0)\in [1,\infty)$ such that 
\begin{equation}\label{HKonup**}
p_t(x,y)\le c_1 t^{-d/2}\qquad \forall x,y\in B(x_0, t^{(1+\delta)/2}),~ t\ge T_0(x_0).\end{equation}
(ii) There exist $\delta> 0, c_2>0$ and $R_0(x_0)\in [1,\infty)$ such that the following hold:

(CSRW case: $\theta=\pi$)  $c_2r^2\le E^{x}[\tau_{B(x,r)}]$ 
for all $x\in B(x_0,r^{1+\delta})$ and all $r\ge R_0(x_0)$.

(VSRW case: $\theta=\nu$) $c_2r^2\le E^{x}[\tau_{\tilde B(x,r)}]$ 
for all $x\in \tilde B(x_0,r^{1+\delta})$ and all $r\ge \tilde R_0(x_0)$.
\\
(iii) There exist $C_E>0$ and $R_1(x_0)\in [1,\infty)$ such that if $R \ge R_1(x_0)$ and a positive function $h : \overline B(x_0,R) \longrightarrow \R_+$ is harmonic on $B = B(x_0, R)$, then writing $B' = B(x_0, R/2)$,  
\begin{equation}
\tag{H}\label{EH}
\sup_{B'} h \le C_E \inf_{B'} h.
\end{equation}
(iv) Let $\theta$ be as above. There exist $\delta>0$, $c_3,c_4>0$ and $R_2(x_0)\in [1,\infty)$ such that 
\begin{eqnarray*}
c_3R^d\le \theta(B(x_0,R))\le \sup_{x\in B(x_0,R^{1+\delta})}\theta(B(x,R))\le c_4R^d,&&
\mbox{for all }~~R\ge R_2(x_0).
\end{eqnarray*} 
(v) (CSRW case: $\theta=\pi$) There exist $\kappa > 0$ and $R_3(x_0)\in [1,\infty)$ such that 
\[
\min_{x\in B(x_0, R)} \pi(x)\ge R^{- \kappa}\,~\qquad~\mbox{for all  }\,~~R\ge R_3(x_0). 
\]
(VSRW case: $\theta=\nu$) There exist $c_5>0$ and $R_4(x_0)\in [1,\infty)$ such that for any $x\in B(x_0,R)$, $R\ge R_4(x_0)$, 
if $d(x,y)\ge R$ then it holds that
\[
\tilde d(x,y)\ge c_5 d(x,y).
\]
\end{assump}

Under the assumption, we have the following.\vspace{.4cm}

\underline{Heat kernel estimates}\vspace{.1cm}
\begin{proposition}\label{HKupperd}
Assume Assumption \ref{asmp:1-1asp} and let $\eps\in (0,\delta/(1+\delta))$. There exist $c_1,\cdots,c_5>0$ and  
$R_*(x_0)\in [1,\infty)$ such that for $x,y\in G$ and $t>0$, 
if 
\begin{equation}\label{eq:ahoeb1}
c_1(d(x,y)\vee t^{\frac 1{2-\eps}})\ge R_*(x_0),
\end{equation}
and
\begin{equation}\label{eq:ahoeb2}
d(x_0,x)\le c_1(d(x,y)\vee t^{\frac 1{2-\eps}}),
\end{equation}
hold, then 
\begin{eqnarray}
p_t(x,y)&\le & c_2t^{-d/2}\exp\Big(-c_3d(x,y)^2/t\Big)\,~~~\mbox{ for }~~t>d(x,y),\label{eq:ketsur1}\\
p_t(x,y)&\le & c_4\exp\Big(-c_5d(x,y)(1\vee \log (d(x,y)/t))\Big)\,~~~\mbox{ for }~~t\le d(x,y). 
\label{eq:ketsur2}
\end{eqnarray}
\end{proposition}

\begin{corollary}\label{thm:corhke} 
Assume Assumption \ref{asmp:1-1asp}. There exist $c_1>0$ and  
$R_*(x_0)\in [1,\infty)$ such that if $R\ge R_*(x_0)$, then
\[
\sup_{0<s\le T}p_s(x,y)\le c_1T^{-d/2}\,\qquad~\mbox{for all }~x,y\in B(x_0,2R)~\mbox{ with }~
d(x,y)\ge R, 
\]
where $T=R^2$.
\end{corollary}
 
For a subset $A\subset G$, let $\{X_t^A\}_{t\ge 0}$ be the process killed on exiting $A$ and
define the Dirichlet heat kernel $p^A_t(\cdot,\cdot)$ as
\[
p_t^A(x,y)=P^x(X^A_t=y)/\theta_y.
\]
Then the following heat kernel lower bound holds. 
\begin{proposition}
\label{Pr:13.1}
Assume Assumption \ref{asmp:1-1asp}. Then there exist $c_1,\delta_0\in (0,1)$ 
and $T_1(x_0)\in [1,\infty)$ such that 
\[
p_t^{B(x_0,t^{1/2})}(x,y)
\ge c_1t^{-d/2},~
\forall x,y\in B(x_0, \delta_0t^{1/2})
\]
for all $t\ge T_1(x_0)$. 
\end{proposition}

\underline{Parabolic Harnack inequalities and H\"older continuity of caloric functions}\vspace{.1cm}
 
\noindent For $x\in G$ and $R, T>0$,
let $C_*\ge 2$, $Q(x,R,T):=(0,4T]\times B(x,C_*R)$ and define 
\begin{eqnarray*}
Q_-(x,R,T):=[T,2T]\times B(x,R),~~Q_+(x,R,T):=[3T,4T]\times B(x,R).
\end{eqnarray*}
Let $u(t,x)$ be a function defined on $[0,4T]\times \bar B(x,C_*R)$. We say $u(t,x)$ is caloric on $Q$
if it satisfies the following: for $t\in (0,4T)$ and $y\in B(x,C_*R)$:
\[
\partial_t u(t,y)={\mathcal L}_\theta u(t,y).
\]
We then have the following.
 
\begin{theorem}\label{thm:PHI}{\bf (Parabolic Harnack inequalities)}\\
Assume Assumption \ref{asmp:1-1asp}. Then 
there exist $c_1>0, C_*\ge 2$
and $R_5(x_0)\in [1,\infty)$ such that for any $R\ge R_5(x_0)$, and 
any non-negative function $u=u(t,x)$ which is caloric on $Q(x_0, R, R^2)$, it holds that
\begin{equation}\label{eq:PHI-u}
\sup_{(t,x)\in Q_-(x_0, R, R^2)}u(t,x)\le c_1\inf_{(t,x)\in Q_+(x_0, R, R^2)}u(t,x).  
\end{equation}
\end{theorem}

\begin{corollary}\label{thm:equiH}
Assume Assumption \ref{asmp:1-1asp}. Then there exist $c_1,\beta>0$, $C_*\ge 2$ and $R_6(x_0)\in [1,\infty)$ such that
the following holds: For any $R\ge R_6(x_0)$ and $T'\ge R^2+1$, let $R'=\sqrt {T'}$ and suppose that
$u$ is a positive caloric function on $Q(x_0,R',T')$. Then for 
any $x_1,x_2\in B(x_0,R)$ and any $t_1,t_2\in [4(T'-R^2),4T']$, we have
\[
|u(t_1,x_1)-u(t_2,x_2)|\le c_1(R/T'^{1/2})^\beta\sup_{Q_+(x_0,R',T')}u. 
\]
\end{corollary}
\ \par 
\underline{Local central limit theorem} \vspace{.1cm}

\noindent In the following, we write
the Gaussian heat kernel with covariance matrix $\Sigma$ (which is a positive definite 
$d\times d$ matrix) as
\[
k_t(x):=\frac 1{\sqrt{(2\pi t)^d\det \Sigma}}\exp (-\frac{x\Sigma^{-1}x}{2t}). 
\]
When $G={\mathbb Z}^d$, $x_0=0$ and $d\ge 2$, 
if we further assume the invariance principle, we can obtain the following local limit theorem.

\begin{proposition}\label{thm:LCLT}
Assume Assumption \ref{asmp:1-1asp} and the following; 
There exists $c_1>0$ such that $\lim_{R\to \infty} R^{-d}\pi(B(0,R))=c_1$ and 
\begin{eqnarray*}
\lim_{n\to\infty}P^0(n^{-1/2}X_{[nt]}\in H(y,R))=\int_{H(y,R)}k_t(z)dz, \qquad\forall y\in \R^d, R, t>0,
\end{eqnarray*}
where $H(y,R)=y+[-R,R]^d$. Then there exist $a>0$ such that 
for each $T_1,T_2>0$ and each $M>0$, we have
\[
\lim_{n\to\infty}\sup_{|x|\le M}\sup_{t\in [T_1,T_2]}|n^dp^\omega_{n^2t}(0,[nx])-ak_t(x)|=0,\]
where we write 
$[x]=([x_1],\cdots, [x_d])$ for $x=(x_1,\cdots, x_d)\in {\mathbb R}^d$. 
\end{proposition}

\subsection{\textbf{Part II: Models and results}}
In this subsection, we will consider the specific RCM 
with i.i.d. conductances that have polynomial lower tails at zero. 
In Part I, we consider a general weighted graph, but in Part II we consider 
$G=\Z^d$ and the conductance is nearest neighbor and random. 

Let us first define the model precisely (for more information on the RCM, see Biskup \cite{Biskup-review} or Kumagai~\cite{T}). Consider the $d$-dimensional hypercubic lattice $\Z^d$ and let $\E_d$ denote the set of (unordered) nearest-neighbor pairs, called edges or bonds, i.e. $\E_d = \{ \{x, y\} : x, y \in \Z^d, | x - y | = 1\}$. We use the notation $x\sim y$ if $(x,y)\in\E_d$, and $\omega_e = \omega_{xy}= \omega_{yx}$ to denote the random conductance of an edge $e$. Let $(\Omega, {\mathcal F}, \PP)$ be the probability space that governs the randomness of the media. We assume $\{\omega_e: e\in \E_d\}$ to be positive and i.i.d..
We define $\pi$, CSRW, VSRW, their Laplace operators and heat kernels etc. as in Part I. Note that we have two sources of randomness for the Markov chain: the randomness of the media and the randomness of the Markov chain. In order to clarify the randomness of the media, we often put $\omega\in \Omega$.
For example, we denote by $(P^x_\omega,\, x\in\Z^d)$ 
the Markov laws induced by the semigroup $\p^t_\omega:=e^{t\mathcal{L}_{\theta}}$, and by 
$
p^\omega_t ( x, y ) = P^x_\omega ( X_t = y )/\theta (y)
$ the heat kernel. Let $E^x_\omega$ be the expectation with respect to $P^x_\omega$.
As in the last subsection, we use the same notation for CSRW and VSRW when it is clear which Markov chain we are talking about. 

\ \newline 

Our purpose is to investigate the effects of fluctuations in the environment on the behavior of the random walk. We shall in particular get bounds on the long time behavior of the return probability $P^0_\omega (X_t=0)$. 

It is well known that when the conductances are bounded and bounded away from $0$ (the uniformly elliptic case), then the decay of the return probability obeys a standard power law with exponent 
$d/2$. Indeed, the following (much stronger) estimates hold:  there exist constants 
$c_1,\cdots,c_4$ such that for all $x$ and $y$
for all $t\ge d(x,y)$, then 
\begin{equation}\label{eq:GHKq}
c_1 t^{-d/2}\exp(-c_2|x-y|^2/t) \le P^x_\omega( X_t = y) \le c_3 t^{-d/2}\exp(-c_4|x-y|^2/t),
\end{equation}
both for CSRW and VSRW. We refer to Delmotte~\cite{Delmotte}. 

The first sharp results for non-uniformly elliptic conductances were obtained 
independently by Barlow \cite{Barlow} and 
by Mathieu and Remy~\cite{Mathieu-Remy} 
in the case of random walks on super-critical percolation clusters. 
In this case,  
conductances are allowed to take two values only, $0$ and $1$. We assume that  $\PP(\omega_b>0)>p_\cc(d)$, where $p_\cc(d)$ is the critical threshold for bond percolation on~$\Z^d$ and we condition on the event that the origin belongs to the infinite cluster of positive conductances. Mathieu and Remy \cite{Mathieu-Remy} showed that there exists a constant $C$ such that, for almost all realizations of the conductances, for large enough $t$, we have   
\begin{equation}
\label{MR}
\sup_y P^0_\omega (X_t=y) \le C t^{-d/2}\,.
\end{equation} 
Barlow \cite{Barlow} obtained detailed
two-sided Gaussian heat kernel bounds for the 
random walks on super-critical percolation clusters.
Namely, he proved \eqref{eq:GHKq} for all $x,y$ on the 
infinite cluster and for large enough $t$. 

Quite often in statistical mechanics, results in percolation help understanding more general situations through comparison arguments; the present paper is no exception. 

The bounds on the return probability in the percolation case eventually lead to the proof of functional central limit theorems and local C.L.T. . We refer to Sidoravicius and Sznitman~\cite{Sidoravicius-Sznitman}, Berger and Biskup~\cite{BB} and Mathieu and Piatnitski~\cite{Mathieu-Piatnitski} for the percolation model, and Barlow and Hambly \cite{BH-LCLT} for the local C.L.T., and to Mathieu~\cite{QIP}, Biskup and Prescott~\cite{BP}, Barlow and Deuschel~\cite{Barlow-Deuschel}, Andres, Barlow, Deuschel and Hambly~\cite{ABDH} for more general models of random conductances.

In the other direction, examples show that a slow decay of the return probability is possible for random positive conductances. 
In Fontes and Mathieu \cite{Fontes-Mathieu}, the authors computed the annealed return probability for a model of random walk with positive conductances whose law has a power tail near $0$. They showed a transition from a classical decay like $t^{-d/2}$ to a slower decay. In~\cite{BBHK}, Berger, Biskup, Hoffman and Kozma proved that  for $d\ge5$, given any sequence $\lambda_n\uparrow\infty$, there exists a product law $\PP$ on $(0,\infty)^{\E_d}$ such that 
$$P^0_\omega (X_{n_k}=0) \geq c(\omega)(\lambda_{n_k} n_k)^{-2}$$ 
along a deterministic sequence $(n_k)$, with $c(\omega)>0$ almost surely. In this construction, although the conductances are almost surely positive, their law has a very heavy tail near $0$ of the form $\PP(\omega_{xy}<s) \sim |\log(s)|^{-\theta}$, $\theta>0$. 
(Here we write $f\sim g$ to mean that
${f(t)}/{g(t)}=1+o(1)$ for functions $f$ and $g$.) 

One may then ask for what choice of $\PP$ does the transition from a classical decay with rate $t^{-d/2}$ to a slower decay happens. 
A partial answer to this question is in the papers of  Boukhadra~\twocite{B1}{B2}. 

Let us consider positive and bounded conductances, with a power-law tail near zero: let $\gamma>0$ and assume the following conditions~: for any $e \in  \E_d $, 
\begin{equation}
\label{P}
\tag{P}
\omega_e \in [0,1],\qquad \PP ( \omega_e \le u ) = 
u^{\gamma} ( 1 + o(1) ), \quad u\to 0.
\end{equation}

It is proved in Boukhadra \cite{B2} that, when  \eqref{P} is satisfied with $\gamma>d/2$ , then 
\begin{equation}
\label{B2eq} 
P^0_\omega (X_t=0) = t^{-\frac{d}{2}+o(1)}, 
\end{equation} for almost all environments and as $t$ tends to $+\infty$. 

On the other hand, it is proved in \cite{B1} that, still  
for an environment satisfying \eqref{P}, then 
for $d\ge 5$
\begin{equation} 
\label{min}
P^0_\omega (X_{2n}=0) \geq C(\omega)\, n^{-(2+\delta)},
\end{equation}
 where $\delta =\delta(\gamma)$ is a constant such that $\delta(\gamma)\longrightarrow0$ as $\gamma \to 0$. 
 
 \medskip

The next theorem improves upon \eqref{B2eq} in two respects: first we have extended the domain of admissible values of $\gamma$; secondly and more importantly, we obtain a much sharper upper bound on the return probability, to be compared with \eqref{MR}. 

In this subsection, we use an equivalent and more appropriate definition of the box:  
$$B ( x, n) = x +  [-n,n]^d \cap \Z^d$$ 
for all $x \in \Z^d$ and write $B_n = B (0, n)$.

\begin{theorem}
\label{Th} 
Let $d \ge 2$ and suppose that the conductances $( \omega_e, e \in \E_d )$ are i.i.d. satisfying \eqref{P}. Then we have :

\noindent
$(1)$ For the CSRW, for any $\gamma > \frac{1}{8} \frac{d}{d - 1/2}$, there exist positive constants $\delta, c_1 > 0$ such that $\PP$-a.s. for all $x,y\in B (0, {t^{(1+\delta)/2}})$ and for $t$ large enough, 
\begin{equation}
\label{HKonup}
p^\omega_t (x,y)\le c_1 t^{-d/2}.
\end{equation}

\medskip
\noindent
$(2)$ For the VSRW for any $\gamma > 1/4$, there exist positive constants $\delta', c_2 > 0$ such that $\PP$-a.s. for all $x,y\in B (0, {t^{(1+\delta')/2}})$ and  for $t$ large enough, 
\begin{equation}
\label{HKonup1}
p^\omega_t (x,y)\le c_2 t^{-d/2}.
\end{equation} 
\end{theorem} 

\medskip

Using the results in Part I, we obtain the following.

\begin{theorem}\label{Fithm} 
 Let $\gamma > \frac{1}{8} \frac{d}{d - 1/2}$ for CSRW and $\gamma >1/4$ for VSRW. Then
 the conclusions of Proposition \ref{Pr:13.1} (Heat kernel lower bound), Theorem \ref{thm:PHI} (Parabolic Harnack inequality), Corollary \ref{thm:equiH} (H\"older continuity of caloric functions)
 and Proposition \ref{thm:LCLT} (Local central-limit theorem) hold. 
\end{theorem}

\begin{remarks} 
Let us discuss in what sense the statements in Theorem \ref{Th} are optimal. 

\noindent $(1)$ 
For $d=2$ or $d=3$, the return probability $P^0_\omega (X_t = 0)$ a.s. decays like $t^{-d/2}$ even when our restrictions on $\gamma$ 
are not satisfied (in fact for any choice of i.i.d. positive conductances) as was proved in \cite{BBHK}. 

From Theorem \ref{Th}, we get that $P^0_\omega (X_t = 0)$ also a.s. decays like $t^{-d/2}$ when $\gamma > \frac{1}{8} \frac{d}{d - 1/2}$ (CSRW) 
or $\gamma>\frac 14$ (VSRW). Whether these restrictions on $\gamma$ are optimal or not, we do not know - but, as recalled in (\ref{min}), 
we know that when $d\ge 5$, then the return probability does not decay like $t^{-d/2}$ for small positive values of $\gamma$.

\smallskip
\noindent
$(2)$ In spite of $(1)$ above, the restrictions on $\gamma$ in Theorem \ref{Th} are optimal as far as the decay of 
$\sup_{x\in B_{\sqrt{t}}}p^\omega_t(x,x)$ is concerned. More precisely, we claim that if $\gamma < \frac{1}{8} \frac{d}{d - 1/2}$ (CSRW) 
or $\gamma<\frac 14$ (VSRW), then $\sup_{x\in B_{\sqrt{t}}}p^\omega_t(x,x)$ cannot decay like $t^{-d/2}$. 
The justification of this claim is related to trapping effects on the random walk induced by fluctuations of the conductances. 
These trapping effects depend on the model, CSRW or VSRW. 

The CSRW cannot be trapped on a site but it might be trapped on an edge. Indeed, assume there exists in $B_n$ an edge $e=\{x,y\}$ of conductance of order $1$ that is surrounded by edges of conductances of order $n^{- \mu}$ for some $\mu>0$. 
Starting at $x$, the random walk will oscillate between $x$ and $y$ for a time of order $n^{\mu}$. If we insist that $p^\omega_t (x,x)\le c_1 t^{-d/2}$ when $t$ is of order $n^2$, 
as in Theorem \ref{Th}\,$(1)$, 
this imposes $\mu<2$. 
It is not difficult to see that, under assumption  \eqref{P}, there will $\PP$-a.s exist edges of conductance of order $1$ that are surrounded by edges of conductances smaller than 
$n^{-\mu}$ for all $\mu$ such that $\mu \gamma(4d-2)<d$. Thus we deduce that it is not correct that $p^\omega_t (x,x)$ decays faster than $t^{-d/2}$ uniformly on the box 
$B_{\sqrt t}$ when $\gamma < \frac{1}{8} \frac{d}{d - 1/2}$. 

The VSRW may be trapped on a point: let $x$ be such that all edges containing $x$ have conductances of order $n^{- \mu}$. Then the VSRW will wait for a time of order $n^{\mu}$ before its first jump. Thus the estimate  $p^\omega_t (x,x)\le c_1 t^{-d/2}$ when $t$ is of order $n^2$ cannot hold unless $\mu < 2$. 
It is easy to deduce from that fact that statement (\ref{HKonup1}) in Theorem \ref{Th}\,$(2)$ is false when $\gamma < 1/4$. 

\smallskip
\noindent
$(3)$ One may also compare our estimates with the results in \cite{ADS}. 
In \cite{ADS}, the authors consider stationary environments of random conductances under some integrability conditions. When applied to i.i.d. conductances satisfying \eqref{P}, they obtain heat kernel upper bounds as in Theorem \ref{Th} provided that $\gamma>1/4$ for both models CSRW and VSRW. (See \cite[Proposition 6.3]{ADS} for CSRW. The same argument also works for VSRW, see the discussion in \cite[Remark 1.5]{ADS}.) 

Thus statement $(2)$ in Theorem \ref{Th} is not new but statement $(1)$ improves upon \cite{ADS} 
for the i.i.d. conductances satisfying \eqref{P}. 
Observe also that our strategy strongly differs from the one in \cite{ADS}. The authors of \cite{ADS} first establish elliptic and parabolic Harnack inequalities from Sobolev inequalities, and 
then deduce heat kernel bounds. We approach the problem the other way around: we shall first establish Theorem \ref{Th} using probabilistic arguments (in particular percolation estimates) 
and deduce the Harnack inequality from Theorem \ref{Th}. 
\end{remarks}

The organization of the paper is as follows. 
The proofs of the results in Part I and II are given in Sections \ref{sec2} and \ref{sec3w}
respectively. The key tool in the proof of Theorem~\ref{Th} (the main theorem in Part II) is Proposition \ref{ll}, and its proof is given in Section \ref{twc}. 
The proof of Proposition \ref{ll} requires some preliminary percolation results and spectral gap estimates, which are 
given in Section~\ref{sec:perco} and \ref{sec:sg} respectively. 
Some relatively standard proof is given in Appendix (Section \ref{appen}) for completeness.

\section{\bf Proof of the results in Part I}\label{sec2}

In the following three sections, we prove results in Part I. 
We first give a preliminary lemma. 

\begin{lemma}\label{Lem:3.1}
(i) Assume Assumption \ref{asmp:1-1asp} (i), (iv). Then 
there exists $c_1>0$ and $R_7(x_0)\in [1,\infty)$ such that 
\begin{equation}\label{eq:upimportant831}
E^y[\tau_{B(x,r)}]\le c_1 r^2,
\end{equation}
for all $x\in B(x_0,r^{1+\delta}/2)$, all $y\in G$ and all $r\ge R_7(x_0)$.\\
(ii) Assume Assumption \ref{asmp:1-1asp} (i), (ii), (iv).
Then there exist $c_2>0, p\in (0,1)$ and $R_8(x_0)\in [1,\infty)$ such that 
\begin{equation}\label{eq:upimport234}
P^x(\tau_{B(x,r)}\le t)\le p+c_2 t/r^2,
\end{equation}
for all $x\in B(x_0,r^{1+\delta})$, $t \ge 0$ and all $r\ge R_8(x_0)$. 
\end{lemma}

\begin{proof}
(i) Let $R_7(x_0):=T_0^{1/2}(x_0)\vee R_2(x_0)$. 
For $R>R_7(x_0)$ and $x\in B(x_0, R^{1+\delta})$, if $y,z\in B(x, R)$ and 
$t=c_*R^2$ where $c_*\ge 4$ is chosen later, we have
$x,y,z\in B(x_0, 2R^{1+\delta})\subset B(x_0, t^{(1+\delta)/2})$ and $t\ge T_0$. Thus, by Assumption \ref{asmp:1-1asp} (i), (iv), we have 
\[ P^y(X_t\in B(x,R))=
\sum_{z\in B(x,R)} p(t,y,z)\theta(z)\leq c_1t^{-d/2} \theta(B(x,R)) \leq 
c_1 c_4t^{-d/2}R^d\le \frac 12, \]
where we chose $c_*^{d/2}\ge 2c_1 c_4$. This implies 
\[ P^y(\tau_{B(x,R)}> t)\leq \frac 12. \] 
By the Markov property, for $m$ a positive integer
\[ P^y(\tau_{B(x,R)}> (m+1)t)\leq E^y[P^{Y_{mt}}(\tau_{B(x,R)}> t): \tau_{B(x,R)}>mt]\leq \frac 12 P^y(\tau_{B(x,R)}>mt). \]
By induction,
\[ P^y(\tau_{B(x,R)}>mt)\leq 2^{-m}, \]
and we obtain $E^y[\tau_{B(x,R)}]\le cR^2$.
When $y\notin B(x, R)$, clearly $E^y[\tau_{B(x,R)}]=0$, so the result follows. 
\\
(ii) Write $\tau=\tau_{B(x,r)}$. Using (i) and Assumption \ref{asmp:1-1asp} (ii), we have
\[
c_2r^2\le E^x[\tau]\le t+E^x[1_{\{\tau> t\}}E^{X_t}[\tau]]\le t+c r^2P^x(\tau> t)
\le t+c r^2(1-P^x(\tau\le t)),
\]
for $x\in B(x_0,r^{1+\delta})$, $r\ge R_0(x_0)\vee R_7(x_0)=:R_8(x_0)$. 
Rewriting, we have 
\[P^x(\tau\le t)\le 1-c_2/c + t/(c r^2),\]
and \eqref{eq:upimport234} is proved.  \end{proof} 

The following lemma is from \cite[Lemma 1.1]{baba}.
\begin{lemma}\label{lem:3.2}
Let $\{\xi_i\}_{i=1}^m, H$ be non-negative random variables such that $H\ge \sum_{i=1}^m\xi_i$. 
If the following holds for some $p\in (0,1)$, $a>0$, 
\[
P(\xi_i\le t|\sigma (\xi_1,\cdots , \xi_{i-1}))\le p+at,\,\qquad~t>0,
\]
then
\[
\log P(H\le t)\le 2(amt/p)^{1/2}-m\log (1/p). 
\]
\end{lemma}

Given Lemma \ref{Lem:3.1}, we have the following. 
\begin{proposition}\label{pro:3.3}
Assume Assumption \ref{asmp:1-1asp} (i), (ii), (iv), and let $\eps\in (0,\delta/(1+\delta))$. 
Then, there exist $c_1,c_2,c_3>0$ such that the following holds for $\rho, t>0$ that satisfy 
$\rho^{2-\eps}\le t$ and $t/\rho\ge c_1R_8(x_0)$;  
\begin{equation}\label{eq:pro33-ne}
P^x(\tau_{B(x,\rho)}\le t)\le c_2\exp (-c_3\rho^2/t),\,\qquad~\mbox{ for all } x\in B(x_0,\rho).
\end{equation}
\end{proposition}
\begin{proof}
The following argument has been often made for heat kernel upper bounds on fractals. We closely follow 
\cite[Proposition 3.7]{Barlow}.

Let $r=\lfloor \rho/m\rfloor\ge 1$ 
where $m\in {\mathbb N}$ is chosen later. Define inductively 
\[
\sigma_0=0,~~\sigma_i=\inf\{t>\sigma_{i-1}: d(X_{\sigma_{i-1}},X_t)=r\},~~r\ge 1.
\]
Let $\xi_i=\sigma_i-\sigma_{i-1}$ and let ${\mathcal F}_t=\sigma(X_s: s\le t)$ be the filtration of $X$. 
By Lemma \ref{Lem:3.1}, we have
\begin{equation}\label{eq:koele3-1}
P^x(\xi_i< u|{\mathcal F}_{\sigma_{i-1}})\le p+c_1u/r^2
\end{equation}
if $X_{\sigma_{i-1}}\in B(x_0, r^{1+\delta})$, $r\ge R_8(x_0)$ and $u\ge 0$. Note that
$d(x, X_{\sigma_m})=d(X_0, X_{\sigma_m})\le mr \le \rho$ so that 
$\sigma_m\le \tau_{B(x,\rho)}$ and $X_{\sigma_i}\in B(x,\rho)$ for $i=0,1,\cdots, m$. 
Using Lemma \ref{lem:3.2} with $a=c_1/r^2$, we obtain 
\begin{eqnarray}
\log P^x(\tau_{B(x,\rho)}\le t)&\le &\log P^x(\sigma_m \le t)\le 2(c_1mt/(pr^2))^{1/2}-m\log (1/p)\nonumber\\
&\le & -c_2m(1-(c_3tm/\rho^2)^{1/2})\label{eq:addonoe}
\end{eqnarray}
if 
\begin{equation}\label{eq:condobrv}
x\in B(x_0, r^{1+\delta}/2), ~~\rho\le r^{1+\delta}/2~~\mbox{ and }~~r\ge R_8(x_0).
\end{equation}
Let $\lambda=\rho^2/(2c_3t)$. If $\lambda\le1$, then \eqref{eq:pro33-ne} is immediate by adjusting 
$c_2$ in \eqref{eq:pro33-ne} appropriately, so we may assume $\lambda >1$. 
If we can choose $m\in {\mathbb N}$ with $\lambda/2\le m<\lambda$
and \eqref{eq:addonoe} 
hold, then we have the desired estimate. 
So let us now verify the conditions \eqref{eq:condobrv}. 
Set $m=\lfloor\lambda/2\rfloor+1
\in [\lambda/2,\lambda)$;
then since $m\ge 1$, we have $r\le \rho$. By definition, 
$r=\lfloor \rho/m\rfloor\ge c_4t/\rho$ for some $c_4>0$, so 
the assumption implies $r\ge c_5R_8(x_0)$. The assumption $\rho^{2-\eps}\le t$ and the fact 
$\eps\in (0,\delta/(1+\delta))$ implies (noting that one can choose 
$\rho\ge r$ large) $r^{1+\delta}>2\rho$. Since $x\in B(x_0,\rho)$, we have verified that
\eqref{eq:pro33-ne} holds.  
\end{proof}

Let $d_\theta(\cdot,\cdot)$ be a metric that satisfies 
\begin{equation}\label{eq:dtheas}
\theta_x^{-1}\sum_yd_\theta(x,y)^2\omega_{xy}\le 1 \,~~~\mbox{ for all }\,~x\in G,
\end{equation}
and 
$d_\theta(x,y)\le 1$ for all $x\sim y\in G$. 
The following estimates, which are generalizations of \cite[Corollary 11, 12]{Dav93},
are given in \cite[Theorem 2.1, 2.2]{Fol}. 

\begin{proposition} 
There exist $c_1,\cdots, c_4>0$ such that the following hold for $x,y\in G$; 
\begin{eqnarray}
p_t(x,y)&\le & \frac{c_1}{\sqrt{\theta_x\theta_y}}\exp\Big(-c_2d_\theta(x,y)^2/t\Big)\,~~~\mbox{ for }~~t>d_\theta(x,y),\label{eq:ketqq1}\\
p_t(x,y)&\le & \frac{c_3}{\sqrt{\theta_x\theta_y}}\exp\Big(-c_4d_\theta(x,y)(1\vee \log (d_\theta(x,y)/t))\Big)\,~~~\mbox{ for }~~t\le d_\theta(x,y). 
\label{eq:ketqq2}
\end{eqnarray}
\end{proposition}

We are now ready to prove Proposition \ref{HKupperd}. 

\begin{proofsect}{Proof of Proposition \ref{HKupperd}}~~~
We first consider CSRW, namely $\theta_x=\pi(x)$. In this case the graph distance 
$d(\cdot, \cdot)$
satisfies the condition of $d_\theta$ in \eqref{eq:dtheas}. 
Write $D=d(x,y)$ and $R=d(x_0,x)$.

Case 1: Consider first the case $D^{2-\eps}\ge t$. By \eqref{eq:ahoeb1},
we have $c_1D\ge R_*(x_0)$, and by \eqref{eq:ahoeb2}, $R\le c_1D$. So
\[
d(x_0,y)\le d(x_0,x)+d(x,y)=R+D\le (c_1+1)D.
\]
Substituting $(c_1+1)D$ to $R$ in Assumption \ref{asmp:1-1asp} (v), we have $\min_{x\in B(x_0, (c_1+1)D)}
\pi(x)\ge c_2 D^{- \kappa}$ if $(c_1+1)D\ge R_3(x_0)$, so taking 
$R_*(x_0)\ge c_1R_3(x_0)/(c_1+1)$ and plugging this into \eqref{eq:ketqq1} and \eqref{eq:ketqq2} gives the desired estimates by noting 
\[
D^\kappa t^{d/2}\le D^{\kappa+d(2-\eps)/2}\le c_3\exp (c_4D^{\eps})\le c_3\exp (c_4D^2/t),~~~\mbox{ for }~
D^{2-\eps}\ge t,
\]
with $c_4>0$ smaller than $c_2/2$ in \eqref{eq:ketqq1}.

\indent Case 2: Consider the case $D^{2-\eps}< t$ and
let $\rho=\lfloor D/2\rfloor+1$ if $D\ge 1$, $\rho=0$ if $D=0$. Note that $d(x_0,y)\le (2D)\vee (2R)$. 
By \eqref{eq:ahoeb1}, $R_*(x_0)\le c_1t^{1/(2-\eps)}$. 
Also by \eqref{eq:ahoeb2}, $R\le c_1t^{1/(2-\eps)}$, so that $d(x_0,y)\le c_5t^{1/(2-\eps)}
<t^{(1+\delta)/2}$ by the choice of $\eps$.  
Since $D^{2-\eps}< t$, $(t/2)/\rho> c_6t^{(1-\eps)/(2-\eps)}$, which is larger than 
$c_6(R_*(x_0)/c_1)^{1-\eps}$. So the assumption for Proposition \ref{pro:3.3} is satisfied by 
choosing $R_*(x_0)\ge c_*R_8(x_0)^{1/(1-\eps)}$ for large $c_*>0$. 
Let $A_x=\{z\in G: d(x,z)\le d(y,z)\}$ and $A_y=G\setminus A_x$. Then
\begin{eqnarray}
p_t(x,y)&=&P^x(X_t=y, X_{t/2}\in A_y)/\theta_y+P^x(X_t=y, X_{t/2}\in A_x)/\theta_y\nonumber\\
&=&P^x(X_t=y, X_{t/2}\in A_y)/\theta_y+P^y(X_t=x, X_{t/2}\in A_x)/\theta_x
=:I+II.~~\mbox{ }~~~\label{eq:symptf}
\end{eqnarray}
Write $\tau=\tau_{B(x,\rho)}$. Then  
\begin{eqnarray*}
I=P^x(X_t=y, X_{t/2}\in A_y)/\theta_y&=&
P^x(\tau<t/2, X_t=y, X_{t/2}\in A_y)/\theta_y\\
&\le & P^x(1_{\{\tau<t/2\}}P^{X_\tau} (X_{t-\tau}=y))/\theta_y\\
&\le & P^x(\tau<t/2)\sup_{z\in \partial B(x,\rho), s< t/2}p_{t-s}(z,y)\\
&\le & c_7\sup_{z\in \partial B(x,\rho), s< t/2}p_{t-s}(z,y)\exp(-c_8D^2/t),
\end{eqnarray*}
where Proposition \ref{pro:3.3} is used in the last inequality. 
Noting that $d(x_0,z)\le R+\rho\le c_8t^{1/(2-\eps)}<t^{(1+\delta)/2}$, 
we obtain $I\le c_9t^{-d/2}\exp(-c_{8}D^2/t)$. 
$II$ can be bounded similarly, so that we obtain \eqref{eq:ketsur1}. 

We next discuss the VSRW case (i.e. $\theta_x=1$) briefly. In this case the metric 
$\tilde d(\cdot, \cdot)/\sqrt M$, where $M$ is the maximum degree of the vertices,   
is relevant; indeed it satisfies the condition of $d_\theta$ in \eqref{eq:dtheas}. 
So the conclusion (w.r.t. $\tilde d$) holds 
if \eqref{eq:ahoeb1} and \eqref{eq:ahoeb2} hold w.r.t. $\tilde d$. 
Using Assumption \ref{asmp:1-1asp} (v), it is easy to verify that 
\eqref{eq:ahoeb1} and \eqref{eq:ahoeb2} w.r.t. $d$ imply  
\eqref{eq:ahoeb1} and \eqref{eq:ahoeb2} w.r.t. $\tilde d$. Finally 
let us deduce \eqref{eq:ketsur1} and \eqref{eq:ketsur2} for $d$ from those for $\tilde d$.
When $t\ge \tilde d(x,y)^2$, \eqref{eq:ketsur1} is an on-diagonal estimate, so no distance appears there.
When $t< \tilde d(x,y)^2$, \eqref{eq:ahoeb1} for $\tilde d$ implies $R_*(x_0)\le c_1\tilde d(x,y)^{1/(1-\eps)}$,
so by taking $(R_*(x_0)/c_3)^{1-\eps}\ge R_4(x_0)$, we can apply Assumption \ref{asmp:1-1asp} (v) (since
$\tilde d(x,y)\le d(x,y)$) and deduce \eqref{eq:ketsur1} and \eqref{eq:ketsur2} for $d$ from those for $\tilde d$. 
Thus the desired estimates are established. 
\end{proofsect}
\begin{remark}
A Gaussian off-diagonal upper bound similar to \eqref{eq:ketsur1} in Proposition \ref{HKupperd} can sometimes be 
deduced from the on-diagonal upper bound using the strategy initiated by Grigor'yan for manifolds \cite{Grig} and developed in \cite{CGZ,Fol,xchen}  for the graph setting. Indeed, motivated by the present paper, the author of \cite{xchen} included in the latest version of his preprint,  
stronger statements (than in the first version of 
the preprint) 
on getting Gaussian off-diagonal upper bounds 
from the on-diagonal decay of the return probabilities at both end-points. With this revised version, one can obtain 
Proposition \ref{HKupperd} as well, but we think it is still worth providing a complete proof of Proposition \ref{HKupperd} 
based on our different approach. 
\end{remark}

\begin{proofsect}{Proof of Corollary \ref{thm:corhke}}
It is easy to check \eqref{eq:ahoeb1} and \eqref{eq:ahoeb2}, so we can apply Proposition \ref{HKupperd}.
If $s\ge R$, then the result follows directly from \eqref{eq:ketsur1}. If $s<R$, then \eqref{eq:ketsur2}
implies 
\[
p_s(x,y)\le c_1\exp\Big(-c_2R(1\vee \log (R/s))\Big)\le c_1\exp(-c_2R)\le c_3R^{-d},
\]
so the result holds.  
\end{proofsect}

We next prepare some propositions in order to prove Proposition \ref{Pr:13.1}.
The idea of the proof is based on that of \cite[Theorem 3.1]{gt1}. 

A function $u$ is said to be harmonic in a set $A\subset\Z^d$ if $u$ is defined in $\overline A$ (that consists of all points in $A$ and all their neighbors) and if $\LL u(x) = 0$ for any $x \in A$.

As a first step, we should check the elliptic oscillation inequalities. 
For any nonempty finite set $U$ and a function $u$ on $U$ , denote
$$\osc_{U}\, u := \max_U u - \min_U u.$$

\begin{proposition}
\label{Pr:11.1}
Assume Assumption \ref{asmp:1-1asp} (iii). Then, for any $\varepsilon > 0$, there exists $\sigma = \sigma(\varepsilon, C_E) < 1$ such that, for any $\sigma R> R_1(x_0)$
and for any function $u$ defined in $\bar B(x_0, R)$ and harmonic in $B(x_0, R)$, we have
\begin{equation}
\label{Pr:11.1-eq}
\osc_{B(x_0,\sigma r)}\, u \le \varepsilon \osc_{B(x_0,r)}\, u,\,
~~~~ \forall r\in (R_1(x_0)/\sigma, R/2].
\end{equation}
\end{proposition}
\noindent The proof is standard. For completeness, we give the proof in Section \ref{appen}.

We write 
$$\bar E(x, R) := \max_{y\in B(x,R)} 
E^y[\tau_{B(x,R)}].$$ 
Then, under Assumption \ref{asmp:1-1asp} (i) and (iv), we have the following due to Lemma \ref{Lem:3.1}:  
\begin{equation}\label{eq:exitbiw}
\bar E (x,R) \le CR^2,\qquad
\forall x\in B(x_0,R), R\ge R_7(x_0).
\end{equation}

The next proposition can be proved similarly as \cite[Proposition 11.2]{gt2}. 
For completeness, we give the proof in Section \ref{appen}.

\begin{proposition}
\label{Pr:11.2}
Assume Assumption \ref{asmp:1-1asp} (iii) and let $R\ge R_1(x_0)$, $u$ be a function 
on $B(x_0, R)$ satisfying the equation $\LL u = f$ with zero boundary condition. Then, for any positive $r < R/2$ with $\sigma r\ge R_1(x_0)$,\begin{equation}
\label{11.2-eq}
\osc_{B(x_0,\sigma r)} u \le 2\big(\overline E (x_0,r) + \varepsilon \overline E(x_0,R)\big)  \max_{B(x_0,R)} |f|,
\end{equation}
where $\sigma$ and $\varepsilon$ are the same as in Proposition~\ref{Pr:11.1}.
\end{proposition}

We now give some time derivative properties of the heat kernel. 

\begin{proposition}
\label{PR:12.1}
Let $A$ be a nonempty finite subset of $\Z^d$.\\
(i) Let $f$ be a function on $A$.
$$u_t(x) =  P^A_t f (x).$$
Then, for all $0 < s \le t$,
\begin{equation}
\label{12.1-eq}
\|\partial_t\, u_t \|_2 \le \frac1s \| u_{t-s}\|_2.
\end{equation}
(ii) For all $x, y \in A$,
\begin{equation}
\label{Pr:12.2-eq}
\left| \partial_t\, p^A_t(x,y) \right| \le \frac1s \sqrt{ p^A_{2v}(x,x)\, p^A_{2(t-s-v)}(y,y)} 
\end{equation}
for all positive $t,s,v$ such that $s+v\le t$.\\
(iii) Under Assumption \ref{asmp:1-1asp} (i), for all $x,y$ we have
\begin{equation}
\label{Pr:12.3-eq}
\left| \partial_t\, p_t^A (x,y) \right|\vee \left| \partial_t\, p_t (x,y) \right| \le C \,t^{-(\frac d2+1)},\quad 
\forall x,y\in B(x_0, t^{(1+\delta)/2}),\, t\ge T_0(x_0).
\end{equation}
\end{proposition}
\noindent The proof is an easy modification of the corresponding results in \cite{gt2} for discrete time. 
For completeness, we give the proof in Section \ref{appen}. 

We are now ready to prove Proposition \ref{Pr:13.1}. 

\begin{proofsect}{Proof of Proposition \ref{Pr:13.1}} 
Let $\varepsilon<1/2$ (we will impose some further bounds of $\varepsilon$ later). Let $R=(t/\varepsilon)^{1/2}$, 
$A = B(x_0,R)$ and for any $x \in B(x_0, \eps R)=B(x_0, (\eps t)^{1/2})$, introduce the function
$$u(y) := p^A_t (x,y).$$
First, we claim that $u(x)\ge c t^{-d/2}$ for large $t>0$. Let $B=B(x,\eps^{1/4}R)$; we choose $\eps$
small enough so that $B\subset A$. Using the Schwarz inequality, we have
\begin{eqnarray*} 
p_t^A(x,x)&\ge &p_t^B(x,x)\ge (\sum_zp_{t/2}^B(x,z)\theta_z)^2/\theta(B)
=(1-P^x(X_t\notin B))^2/\theta(B)\\
&\ge & (1-P^x \big( \tau_{B(x,\eps^{1/4}R)}\le t))^2/\theta(B)\ge 
(1-p-c_6\varepsilon^{1/2})^2/\theta(B)\ge c/\theta(B), 
\end{eqnarray*}
where \eqref{eq:upimport234} is used in the third inequality and we take $\varepsilon>0$ small enough.
(We take $R$ large so that $\varepsilon^{1/4}R\ge R_8(x_0)$.)  
So, using Assumption \ref{asmp:1-1asp} (iv), the claim follows. 

Now let us show that
\begin{equation}
\label{13.5}
| u(x) - u(y) | \le \frac c2 \,t^{-d/2}
\end{equation}
for all $y\in B(x_0, \eps R)$ so that $d(x,y)\le 2(\eps t)^{1/2}$, which would imply $u(y) \ge (c/2) t^{-d/2}$ and hence prove 
the desired result.

Noting that $x\in B(x_0, \eps R)\subset B(x_0,R)$, by Proposition~\ref{PR:12.1}\,(iii),
\begin{equation}
\label{max-deriv-hk}
\max_{y\in B(x_0,R)} \left| \partial_t\, p^A_t (x,y)\right| \le C\, t^{-(\frac d2+1)},\quad
\text{ for large $t$}.
\end{equation}
By Proposition~\ref{Pr:11.2}, we have, for any $0 < r < R/3$ and for some $\sigma \in (0,1)$,
\begin{equation}
\label{11.2-eq-bis}
\osc_{B(x_0,\sigma r)} u \le 2\big(\overline E (x_0,r) + \varepsilon^2 \overline E(x_0,R)\big)  \max_{y\in B(x_0,R)} \left| \partial_t\, p^A_t (x,y)\right|,
\end{equation}
for all $\sigma r\ge R_1(x_0)$ where $\varepsilon$ in Proposition~\ref{Pr:11.2}
is now written as $\varepsilon^2$.

Estimating $\max \left| \partial_t\, p^A_t (x,y)\right|$ by \eqref{max-deriv-hk} and using \eqref{eq:exitbiw}, we obtain, from \eqref{11.2-eq-bis},
$$\osc_{B(x_0,\sigma r)} u \le C\, \frac{r^2 + \varepsilon^2 R^2}{t^{d/2+1}},\quad \forall x\in B(x_0,r), 
~~~\text{$t,r$ large}.
$$

Choosing $r= \varepsilon R$ and noting $t ={\varepsilon} R^2$, we obtain
\begin{equation}
\label{13.8}
\osc_{B(x_0,\sigma r)} u \le 2C \varepsilon\, t^{-d/2}\le \frac c2 \,t^{-d/2}
\end{equation}
provided $\varepsilon\le c/(4C)$,
$x\in B(x_0,(\varepsilon t)^{1/2})=B(x_0,\varepsilon R)$ and $t$ large.

Note that
$$\sigma r = \sigma \varepsilon R = \sigma \varepsilon \left( \frac t\varepsilon \right)^{1/2}  = \sigma \sqrt\varepsilon\ t^{1/2} = \delta_0 t^{1/2},$$
where $\delta_0= \sigma \varepsilon^{1/2}$.  
Hence \eqref{13.8} implies \eqref{13.5}, which was to be proved.
\end{proofsect}

Let us briefly mention 
other proofs of the results in Part I. 

\noindent Proof of Theorem \ref{thm:PHI} is given in Section \ref{appen}.

\noindent 
\begin{proofsect}{Proof of Corollary \ref{thm:equiH}}  
Given Theorem \ref{thm:PHI}, 
the proof is standard and similar
to the proof of \cite[Corollary 4.2]{BGK}. (Given Theorem \ref{thm:PHI}, one can also modify the proof of \cite[Proposition 4.6]{ADS} and \cite[Proposition 3.2]{BH-LCLT}.) So we omit the proof.
\end{proofsect}

\noindent 
\begin{proofsect}{Proof of Proposition \ref{thm:LCLT}}  
 Given Corollary \ref{thm:equiH}, the proof is 
similar to \cite[Theorem 1.11]{ADS} and \cite[Theorem 4.2]{BH-LCLT}, so we omit it. 
\end{proofsect}

\section{\bf Proof of the results in Part II}\label{sec3w}

\subsection{\bf Strategy and proof of Theorem~\ref{Th}} \label{St-Pr}
We now discuss the strategy of the proof of Theorems~\ref{Th} and how one compares random walks with random conductances with random walks on percolation clusters. 

Choose a threshold parameter $\xi>0$ such that $\PP(\omega_b\geq \xi)>p_c(d)$ where $p_c (d)$ is the threshold percolation cluster. The i.i.d.\ nature of the probability measure~$\PP$ ensures that for $\PP$ almost any environment $\omega$, there 
exists a unique infinite cluster in the graph $(\Z^d,\E_d)$, that we denote by $\scrC^\xi = \scrC^\xi(\omega)$.

Provided $\xi$ is small enough, the complement of $\scrC^\xi$ in $\Z^d$, here denoted by $\scrH^\xi$, is a union of finite connected components that we will refer to as {\it holes}, see Lemma \ref{v-h}. Thus, by definition, holes are connected sub-graphs of the grid. Note that holes may contain edges such that $\omega_b\geq\xi$. 

Consider the following additive functional~:
\begin{equation}
\label{addf}
{A}(t)=\int^t_0\1_{\{X_s\in \scrC^\xi\}} \text ds. 
\end{equation}

We shall need to make a time change for the process $X$ to bring us back to the situation that we already know, namely random walks on an infinite percolation cluster.

Recall $A (t)$ from \eqref{addf} and let $A^{-1}(t)=\inf\{s; A (s)>t\}$ be its inverse. Define the corresponding time changed process
$$
X^\xi_t \deq X_{A^{-1}(t)} ,
$$
which is obtained by suppressing in the trajectory of $X$ all the visits to the
holes.

For the proof of Theorem~\ref{Th}, we need the fact that $X^\xi$ behaves in a standard way in almost any realization of the environment~$\omega$ (see for eg. \cite[Lemma 4.1]{QIP} or \cite[Theorem 4.5]{ABDH}). Recall that we use here the box $B_n = [ - n, n ]^d \cap \Z^d$.
\begin{lemma}
\label{Pr:sb-f}
There exists a constant $c_1$ such that $\PP$-a.s. and for $t$ large enough,
\begin{equation}
\label{sb-f}
\sup_y P^x_\omega ( X^\xi_t = y)\leq  c_1\, {t^{-d/2}}, 
\end{equation} 
for all $x\in B_t\cap \scrC^\xi$.
\end{lemma}

The key tool in the proof of Theorem~\ref{Th} is the following control on the time spent by the process outside $\scrC^\xi$. 

Call $\tau_h$ the exit time of the random walk $X$ from $\scrH^\xi$; if $X_0 \notin \scrH^\xi$, then $\tau_h = 0$. 

\begin{proposition}
\label{ll}
$(1)$ Let $d \ge 2$ and choose $\varepsilon \in ( 0, 1)$. Then,  

\noindent
$(1)$ For the CSRW, for any $\gamma > \frac18 \frac{d}{d-1/2}$, there exist positive constants $\delta, \sigma$ and $c_1, \cdots, c_4$ such that 
for $\xi >0$ small enough, $\PP$-a.s. for all $x \in B(0, t^{(1+\delta)/2})$ and all  $t$ large enough, we have 
\begin{equation}
\label{ll1}
P^x_\omega (A (t) \leq \varepsilon\, t) \le c_1 e^{-  c_2 t^{\sigma}}, 
\end{equation}
and 
\begin{equation}
\label{ll2}
P^x_\omega \left( \tau_h \ge  t/2 \right) \le c_3 e^{-  c_4 t^{\sigma}}.
\end{equation}

\smallskip
\noindent
$(2)$ For the VSRW for any $\gamma > 1/4$, there exist positive constants $\delta', \sigma'$ and $c_5, \cdots, c_8$ such that 
for $\xi >0$ small enough, $\PP$-a.s. for all $x \in B(0, t^{(1+\delta')/2})$ and all  $t$ large enough, we have \begin{equation}
\label{ll3}
P^x_\omega (A (t) \leq \varepsilon\, t) \le c_5 e^{-  c_6 t^{\sigma'}}\qquad \hbox{and}\qquad 
P^x_\omega \left( \tau_h \ge  t/2 \right) \le c_7 e^{-  c_8 t^{\sigma'}}. 
\end{equation}
\end{proposition}

\medskip

\begin{proofsect}{Proof of Theorem~\ref{Th}} 
Let $X$ be the CSRW with conductances satisfying \eqref{P} and assume $\gamma > \frac18 \frac{d}{d-1/2}$. One can follow the same argument for the VSRW with $\gamma > 1/4$ and with the counting measure instead of $\pi$.

We start by reproducing here the same reasoning as in \cite{B2}. 
Let $n = t^{(1+\delta)/2}$ with $\delta$ as in Proposition~\ref{ll} and such that $\delta<1$. Assume first that $x$ belongs to $\scrC^\xi \cap B_n$.
Since the probability of return is decreasing, see for eg. \cite[Lemma 3.1]{B2}, we have
\begin{equation}
\label{fs-ineq}
P^x_\omega(X_t = x )\leq \frac{2}{t}\int^{t}_{t/2}P^x_\omega(X_v = x )\text dv=\frac{2}{t} E^x_\omega\left[\int^{t}_{t/2}\1_{\{X_v = x \}}\text dv\right].
\end{equation}

The additive functional $A ( \cdot )$ being a continuous increasing function of the time and null outside the support of the measure $\text{d} A (v)$, so 
taking $u=A (v)$ and noting that 
$A'(v)=\1_{\{X_v\in\scrC^\xi\}}$, we get 

\begin{eqnarray*}
E^x_\omega\left[\int^{t}_{t/2}\1_{\{X_v =  x \}}\text dv\right]
&=&
E^x_\omega\left[\int^{t}_{{t}/2}\1_{\{X_v = x \}} \1_{\{X_v\in\scrC^\xi\}}\text dv\right]
\\
&=&
E^x_\omega\left[\int^{A (t)}_{A (t/2)}\1_{\{X^\xi_u = x \}}\text du\right], 
\end{eqnarray*}
which is bounded by
$$
E^x_\omega\left[\int^{t}_{A (t/2)}\1_{\{X^\xi_u = x \}}\text du\right],
$$
since $A (t)\leq t$.

Therefore, for $\varepsilon\in(0,1)$
\begin{eqnarray*}
\begin{split}
& P^x_\omega(X_t = x )
\leq \frac{2}{t}E^x_\omega\left[\int^{t}_{A (t/2)}\1_{\{A (t/2) \geq  \varepsilon\, t/2\}}\1_{\{X^\xi_u = x \}}\text du\right]\\
& \qquad\qquad \qquad \qquad\qquad+ \frac{2}{t}E^x_\omega\left[\int^{t}_{A (t/2)}\1_{\{A (t/2) \leq \varepsilon\, t/2\}}\1_{\{X^\xi_u = x \}}\text du\right]\\
& \qquad \qquad\qquad\,\leq 
\frac{2}{t}\int^{t}_{\varepsilon t/2}P^x_\omega(X^\xi_u = x )\text du
+  \frac{2}{t} \int^{t}_{0}P^x_\omega(A (t/2) \leq \varepsilon\, t/2)\text du,
\end{split}
\end{eqnarray*}
and using Lemma~\ref{Pr:sb-f},
\begin{eqnarray}
\label{pb}
P^x_\omega(X_t = x )
&\leq&
\frac{2 c_1}{t}\int^{t}_{\varepsilon t/2}u^{-d/2}\text du + 2 P^x_\omega(A (t/2) \leq \varepsilon\, t/2)  \nonumber
\\
&\leq&\label{ls-ineq}
{2 c_1} ( 1 - (\varepsilon/2)^{1 - d/2} )\, {t^{- {d}/{2}}} + 2 P^x_\omega (A (t/2) \leq \varepsilon\, t/2),
\end{eqnarray} 
which by virtue of Proposition~\ref{ll} for $t$ large enough, is less than 
$$
c_2\, {t^{- {d}/{2}}} + 2 c_3\, e^{- c_4 t^\sigma}.
$$  
Since $\pi ( x ) > \xi$, we obtain that
$$
p^\omega_t ( x, x ) \le c_5\, t^{- d/2}.
$$
Then
Cauchy-Schwarz gives 
\begin{equation}
\label{hk-strcl}
p^\omega_t ( x, y ) \le \sqrt{p^\omega_t ( x, x ) p^\omega_t ( y, y) } \le c_6\, t^{- d/2}, 
\end{equation}
for any $x, y \in B \big( 0, t^{(1+ \delta)/2} \big) \cap \scrC^\xi$ and all $t$ large enough.

Recall $n = t^{(1+ \delta)/2}$. Suppose $x \in \scrH^\xi \cap B_n$ and $y \in \scrC^\xi \cap B_n$. Note that $x$ belongs to a hole with a size less than $( \log n)^c$ included in $B_{2n}$ (see Lemma~\ref{v-h} below). It implies that $X_{\tau_h} \in \scrC^\xi \cap B_{2n}$ if $X_0 =x$. Then the strong Markov property gives
\begin{equation}
P^x_\omega ( X_t = y ) \le P^x_\omega ( \tau_h > t/2 ) 
+ E^x_\omega \left( \1_{\{\tau_h \le t/2\}}\, P^{X_{\tau_h}}_\omega ( X_{t-\tau_h} = y ) \right)
\end{equation}
which, by \eqref{hk-strcl} and \eqref{ll2}, and for $t$ large enough, is less than
\begin{equation}
\label{hk-strwea}
c_3 e^{- c_4 t^{\sigma (1+\delta)/2}} + \max_{z \in \scrC^\xi \cap B_{2n}} \sup_{s \in [ t/2, t]} P^z_\omega ( X_s = y ) 
\le c_7\, t^{- d/2}\pi(y). 
\end{equation} 
Since $\pi ( y ) \ge \xi$, we deduce that 

\begin{equation}
\label{hk-1strongsite}
p^\omega_t (x, y)  \le c_{8} t^{- d/2}.
\end{equation} 
Using the reversibility, we also deduce that 
\begin{equation}
\label{hk-1strongsite'}
p^\omega_t (x,y)  \le c_{9} t^{- d/2}
\end{equation} 
whenever $y \in \scrH^\xi \cap B_n$ and $x \in \scrC^\xi \cap B_n$. 

Last, suppose $x, y \in \scrH^\xi \cap B_n$. The strong Markov property yields  
\begin{equation}
\frac{P^x_\omega ( X_t = y ) }{\pi (y)} \le \frac{P^x_\omega ( \tau_h > t/2 )}{\pi (y)} 
+ \frac{1}{\pi (y)}\, P^x_\omega \left( \1_{\{\tau_h \le t/2\}}\, P^{X_{\tau_h}}_\omega ( X_{t-\tau_h} = y ) \right),
\end{equation} 
which by \eqref{ll2} and \eqref{hk-1strongsite'} is less than
\begin{equation}
\frac{c_3}{\pi (y)} \, e^{- c_4 t^{\sigma (1+\delta)/2}} +  \max_{z \in \scrC^\xi \cap B_{2n}} \sup_{s \in [t/2, t]} p_s ( z, y)
\le 
\frac{c_3}{\pi (y)} \, e^{- c_4 t^{\sigma (1+\delta)/2}} + c\, t^{- d/2}.
\end{equation}
Since $1/\pi (y) \le n^c$ with a constant $c$ depending only on $d$ and $\gamma$ (cf. Lemma \ref{fontes-mathieu} below), the claim follows.
\end{proofsect}

The proof of Proposition \ref{ll}  is deferred to Section \ref{twc}. Section~\ref{sec:perco} contains some preliminary percolation results, followed by Section~\ref{sec:sg}, which provides some spectral gap estimates necessary to the proof of the proposition. 

Although the main strategy is close to the argument in Boukhadra \cite{B2}, note that the spectral gap estimates we prove here are sharper and their proof involves a much more detailed analysis of the geometry of the percolation cluster. 

 \subsection{\bf Proof of Theorem \ref{Fithm}}
\begin{proofsect}{Proof of Theorem \ref{Fithm}}
It is enough to check Assumption \ref{asmp:1-1asp} with $x_0=0$ and the hypothesis in 
Proposition \ref{thm:LCLT}.
\eqref{HKonup**} is a consequence of Theorem \ref{Th}. Assumption \ref{asmp:1-1asp} (ii) holds since it 
is true for the time changed process $X^\xi$  as in \cite[Proposition 4.7]{ABDH}. 
\eqref{EH} is proved in \cite[Theorem 7.3]{ABDH}.
Note that VSRW and CSRW share the same harmonic functions, so this fact can be used both of them. 
Assumption \ref{asmp:1-1asp} (iv) will be proved in Lemma \ref{lm:V} for the CSRW case (it is trivial 
for the VSRW case because the reference measure is a uniform measure). 
Assumption \ref{asmp:1-1asp} (v) for CSRW case is true because of Lemma~\ref{fontes-mathieu} below. 
For $n$ large enough (larger than some random integer), we have
$$
\min_{x \in B_n} \pi (x) \ge n^{- \kappa}\qquad \text{with}\quad \kappa > \frac{d}{\gamma},
$$
where $\gamma$ is the parameter that we see in the law of the environment~\eqref{P}.
Assumption \ref{asmp:1-1asp} (ii), (v) for VSRW case is obvious in this case because 
$\tilde d (\cdot,\cdot)=d(\cdot,\cdot)$ in our setting since $\omega_e\le 1$ for each edge.

The first hypothesis in Proposition \ref{thm:LCLT} holds by the law of large numbers, and the
second hypothesis is proved 
in \cite[Theorem 2.1]{BP} and \cite[Theorem 1.3]{QIP}.
\end{proofsect}

\section{\textbf{Percolation}}
\label{sec:perco}

This section contains percolation results necessary to the spectral gap estimates in the following section.

We consider the standard Bernoulli percolation model on the grid $\Z^d$: we independently assign to edges the value $1$ (open) and $0$ (closed) with probability 
$p$ and $q=1-p$. 
Let $\PP$ denote the product probability measure thus defined on $\{0,1\}^{\E_d}$. 
We assume $p$ is supercritical so that, for $\PP$ almost any environment $\omega$, there exists a unique infinite open cluster that we denote by $\scrC$. 
For $q$ 
small enough, the complement of $\scrC$ in $\Z^d$, denoted by $\scrH$, is a union of finite open clusters that are called \textit{holes}. 

Let $x\in\Z^d$ and let $\scrH_x$ be the (possibly empty) set of sites in the finite component of $\Z^d\setminus\scrC$ 
containing $x$.

\begin{lemma}
\label{v-h}
Let $d\ge2$. For $p$ sufficiently close to $1$, 
there exist constants $C < \infty$ and $c>0$ such that for all $n\ge1$
$$\PP ( \diam\, {\scrH}_0 > n ) \le C\, e^{- c n}.$$
Here `` $\diam$" is the diameter in the $|\,\cdot\,|_\infty-$distance on $\Z^d$.
\end{lemma}
\begin{proof} See Lemma~3.1 in \cite{BP}. 
\end{proof}

Recall $B_n = [-n,n]^d \cap \Z^d$ the ball in $\Z^d$ centered at $0$ and of radius $n$. We have the following lemma on the proportion of sites belonging to $\scrC$ in a box $B_n$.

\begin{lemma}
\label{density}
Let $\eta \in(0, 1)$. For $p$ sufficiently close to $1$, there exists constants $C < \infty$ and $c>0$ such that for all $n\ge1$ 
\begin{equation}\label{density-eq}
\PP \big(| B_n \cap \scrC  | \le \eta |B_n|\big) \le Ce^{- c n}.
\end{equation}  
\end{lemma}

This estimate 
is sufficient for us, but we do not think it is optimal. 
The expected behavior would be an exponential decay in the perimeter of $B_n$ as in dimension $2$, \cite[Theorem 3] {durrett}.

\begin{proof}  Let $\theta_d(p)$ be the {bond percolation probability} in the grid $\Z^{d}$. Note that $\theta_d(p)$ tends to $1$ when $p\rightarrow1$ [cf. \cite{G}, Section~1.4]. Call $\scrC(\G)$ the infinite percolation cluster of a (sub) graph $\G\subseteq\Z^d$.

First note that $\PP-$a.s.

\begin{equation}
S_n \deq \sum_{x \in B_n} \1_{\{x \in\scrC\}} \ge \sum_{-n\le \ell \le n}\, \sum_{x \in \{\ell\}\times [-n,n]^{d-1}} \1_{x \in \scrC(\{\ell\}\times \Z^{d-1})} =\mathrel{\mathop:} \sum_{-n\le \ell \le n} S_n(\ell). 
\end{equation}

\noindent
Then repeating the operation we get
\begin{equation}
\label{sum}
S_n \ge \sum_{-n\le \ell_1,\ldots,\ell_{d-2}\le n} S_n(\ell_1,\ldots,\ell_{d-2}) 
\end{equation}

\noindent
with
$$S_n (\ell_1,\ldots,\ell_{d-2}) \deq \sum_{x \in \prod_{i=1}^{d-2} \{\ell_i\}\times [-n,n]^{2}} \1_{x \in \scrC(\{\ell_1\}\times\cdots\times\{\ell_{d-2}\}\times \Z^{2})}.$$
The sub-graphs $\{\ell_1\}\times\cdots\times\{\ell_{d-2}\}\times \Z^{2}$ are disjoint copies of $\Z^2$ in $\Z^d$.
 
Now set 

$$Y(\ell_1,\ldots,\ell_{d-2}) \deq S_n(\ell_1,\ldots,\ell_{d-2}) / (2n+1)^2.$$

\noindent 
Let $\eta \in (0, 1)$ and choose $p$ sufficiently close to $1$ such that $\eta \in(0, \theta_2(p))$. By \cite[Theorem 3]{durrett}, for any $\ell_1,\ldots,\ell_{d-2} \in [-n,n]$ and for some $c, C>0$, we have 
\begin{equation}
\PP (Y ( \ell_1,\ldots,\ell_{d-2}) \le \eta ) \le C  e^{- c n}. 
\end{equation} 
Combined  with \eqref{sum}, it implies that 

\begin{eqnarray*}
\PP \big(| B_n \cap \scrC  | \le \eta |B_n| \big) 
&\le&
\PP \Big(\sum_{-n\le \ell_1,\ldots,\ell_{d-2}\le n} Y(\ell_1,\ldots,\ell_{d-2} ) / (2n+1)^{d-2} \le \eta  \Big)\nonumber\\
&\le&
\PP \Big( \bigcup_{-n\le \ell_1,\ldots,\ell_{d-2}\le n} \big\{ Y(\ell_1,\ldots,\ell_{d-2} ) \le \eta \big\} \Big)\nonumber\\
&\le&
C\, n^{d-2}  e^{- c n},
\end{eqnarray*}
which gives \eqref{density-eq}.
\end{proof} 

Write $\CC ( x )$  for the open cluster containing the point $x$.  Then we have: 

\begin{lemma}
\label{nper}
For $q$ small enough, there exists a constant $c_1>1$ such that  
\begin{equation}
\label{nper-eq} 
\PP ( | \CC (0) | < \infty ) \le c_1 \, q^{2d}, 
\end{equation}
and, for all $x\sim0$, 
\begin{equation}
\label{nper-eq'} 
\PP ( | \CC (0) | < \infty \text{\,and\,} | \CC (x) | < \infty ) \le c_1 \, q^{4d - 2}.
\end{equation}

\end{lemma}

\begin{proof} 
Let us recall some necessary definitions that we can find in \cite{G}, Section 1.4. 
Call a \textit{plaquette} any unit $(d-1)$-dimensional hypercube in $\R^d$ that is a face of a cube of the form 
$x+[-\frac 12,\frac 12]^d$. Let $\LLL_d$ be the set of plaquettes. There is a one to one correspondence between edges in $\E_d$ and plaquettes in $\LLL_d$. 
Indeed, for any edge $\{x,y\}\in\E_d$, the segment $[x,y]$ intersects one and only one plaquette. 
We say a set of plaquettes is connected if all   
plaquettes in the set are connected by bonds
in the dual lattice of $\Z^d$. 

We couple the percolation process on $\E_d$ with a percolation on $\LLL_d$ by declaring a plaquette \textit{open} when the corresponding edge is open and 
declaring it is \textit{closed} otherwise.   

Let us suppose that $\CC(0)$ is finite. Then there exists a finite \textit{cutset} of closed plaquettes, say $\varpi$,  around the origin. 
(A cutset around the origin is a 
connected set of plaquettes 
$\mathtt c$ such that the origin  lies in a finite connected component of the complement of $\mathtt c$.) 

The number of such cutsets around the origin which contain $m$ plaquettes is at most $\mu^m$, for some constant $\mu = \mu(d)$ depending only on the dimension. 
The smallest cutset is unique and contains $2d$ plaquettes.  
Then the usual `Peierls argument' gives that the probability on the left hand side in \eqref{nper-eq} is bounded by

\begin{equation*}
\sum_{\varpi\text{,\,cutset around $0$}} \PP (\text{all plaquettes in\,} \varpi\, \text{are closed} ) \le \sum_{m \ge 2d} (\mu q)^m,
\end{equation*}
which converges and is bounded by  $c q^{2d}$ for some $c$ 
provided $p$ is sufficiently close to 1 such that $q \mu < 1$. 

As for the second estimate \eqref{nper-eq'}, we follow the same argument but we find the exponent $4 d - 2$ since this is the size of the smallest number of plaquettes necessary 
to form a cutset around both the origin and $x$. 
\end{proof} 

We now describe application of the preceding lemmas to conductances satisfying assumption \eqref{P}. 

We recall the following result. Call $\B_n$ the set of edges in the box $B_n$.

\begin{lemma} 
\label{fontes-mathieu}
Suppose that the conductances $( \omega_e, e \in \E_d)$ satisfy \eqref{P}. Then $\PP$-a.s., we have
$$\lim_{n \rightarrow \infty} \frac {\log \inf_{e \in \B_n} \omega_e} {\log n}= -\frac d\gamma.$$
\end{lemma}
\begin{proof} The proof is similar to \cite[Lemma 3.6]{Fontes-Mathieu}.\end{proof} 

\noindent 
The density estimate~Lemma~\ref{density} yields the following volume property for the measure $\theta$.

\begin{lemma}
\label{lm:V}
Let $\eta \in (0, 1)$ and $\beta\ge 1$. Let $\PP$ be a product probability measure satisfying \eqref{P}. Then for $\xi > 0$ small enough, 
for $\PP$-a.e. environment, for all $x \in B_{n^\beta}$ and $n$ large enough, we have
\begin{equation}\label{V}
\xi \, \eta | B_n | \le \pi \big( B ( x, n) \big) \le 2d \, | B_n |.
\end{equation}
\end{lemma}
\begin{proof}
Let $\eta \in (0, 1)$. Recall the infinite cluster $\scrC^\xi$ introduced in subsection~\ref{St-Pr}. The right-hand side inequality in \eqref{V} comes from the fact that $\pi (x) \le 2d$. As for the left-hand side inequality, observe that by \eqref{density-eq} and the i.i.d. character of the conductances, 
\begin{equation}
\PP \Big( \bigcup_{x \in B_{n^\beta}} \big\{ \big|\scrC^\xi \cap  B ( x, n) \big| < C \eta | B_n | \big\} \Big) \le | B_{n^\beta} | \, e^{- c n}.
\end{equation}
By the Borel-Cantelli lemma, we get  that for $n$ large enough, for all $x \in B_{n^\beta}$, 
we have $\big|\scrC^\xi \cap  B ( x, n) \big| \ge \eta | B_n |$.  
Since $\pi (x) \ge \xi$ for $x \in \scrC^\xi$, the claim follows.
\end{proof}

In next two lemmas we construct sets of `good' paths in the percolation clusters.

\begin{lemma}
\label{good path} 
Let $\PP$ be a product probability measure satisfying \eqref{P}. \\ 
$(1)$ Let $\gamma > \frac18 \frac{d}{d-1/2}$ and choose $\alpha \in (0,2)$ such that $\gamma \alpha ( 4d - 2) > d$. For $\xi$ small enough,  $\PP$-a.s., for $n$ large enough, 
for each edge 
$e$ in $B_n$, there exists a path of conductances larger than $n^{-\alpha}$ connecting one of the endpoints of $e$ to the frontier $\partial B_{n}$.

\medskip
\noindent
$(2)$ 
Let $\gamma > {1}/{4}$ and choose $\alpha \in (0,2)$ such that $\gamma > {1}/{(2\alpha)}$. For $\xi$ small enough,  $\PP$-a.s. for $n$ large enough, 
for each 
$x \in B_n$, there exists a path of conductances larger than $n^{-\alpha}$ joining $x$ to the frontier $\partial B_{n}$. 
\end{lemma}

Let $\scrH_n=B_n \cap \scrH^\xi$ and $\scrC_n=B_n \cap \scrC^\xi$. 

\begin{lemma}
\label{h-inj}
$(1)$ Let $\gamma > \frac18 \frac{d}{d-1/2}$ and choose $\alpha \in (0,2)$ such that $\gamma \alpha ( 4d - 2) > d$. For $\xi$ small enough, $\PP$-a.s. for $n$ large enough, there exists an injective map $\varphi$ on $\scrH_n$ into $\scrC_n$ such that 
for each edge $e = \{ x, y \} $ with $x \in \scrH_n$, there exists a path $\ell(e, \varphi (x))$ from one of the endpoints of $e$ to $\varphi(x)$ satisfying 
\begin{equation}
\label{ell}
|\ell( e, \varphi (x))|\le (\log n)^{2 d^2} \quad \text{and}\quad \frac{1}{\omega_b} < 4\, n^\alpha,\quad \forall b \in\ell( e, \varphi (x)).
\end{equation}

\noindent
$(2)$ Let $\gamma > 1/4$ and choose $\alpha \in (0,2)$ such that $2 \alpha \gamma > 1$. For $\xi$ small enough, $\PP$-a.s. for $n$ large enough, there exists an injective map $\varphi'$ on $\scrH_n$ into $\scrC_n$  such that 
for each $x\in\scrH_n$, there exists a path $\ell( x, \varphi' (x))$ from $x$ to $\varphi'(x)$ satisfying 
\begin{equation}
\label{ellbis}
|\ell( x, \varphi' (x))|\le (\log n)^{2 d^2 } \quad \text{and}\quad \frac{1}{\omega_b} < 4\, n^\alpha,\quad \forall b \in\ell( x, \varphi' (x)). 
\end{equation}
\end{lemma}

\begin{proofsect}{Proof of Lemma~\ref{good path}}
$(1)$  Let $\alpha \gamma ( 4 d - 2) > d$ for some $\alpha \in (0,2)$. Recall $\B_n$ the set of edges in the box $B_n$ and set $\partial \B_n = \B_n \setminus \B_{n-1}$. Note that $| \B_n | = O ( n^{d-1} )$.

Let $\scrE_n$ be the event:  there exists an edge $e \in \partial \B_{n}$ such that none of its endpoints can be joined by a path to $\partial \B_{n+1}$ along edges with conductances larger than $n^{-\alpha}$; this last event is denoted by $\{e \nleftrightarrow \partial \B_{n+1}\}$, i.e. 

\begin{equation}
\scrE_n \deq \bigcup_{e \in \partial \B_{n}}  \{e \nleftrightarrow \partial \B_{n+1}\}.
\end{equation}

\noindent 
Then Lemma~\ref{nper} with 
$q = \PP ( \omega_e < n^{- \alpha\gamma} )$ 
and \eqref{P} imply that 

\begin{equation}
\label{contour}
\PP \big( \scrE_n \big) \le c\, n^{- 1 - ( \alpha \gamma ( 4 d - 2) - d)}, 
\end{equation} 
By the Borel-Cantelli lemma we then get that there is a finite positive random variable $N = N(\omega)$ such that for any $n\ge N$, for every edge  $e \in\partial \B_{n}$, there exists a path of conductances larger than $n^{-\alpha}$ joining $e$ to $\partial \B_{n+1}$. It implies that there exists a path of conductances larger than $n^{-\alpha}$ joining one of the endpoints of  every edge in $\B_{n} \setminus \B_{N-1}$ to $\partial \B_{n+1}$. Indeed, consider an edge $f\in\partial \B_{m}$ for some $m\ge N$. From one of its endpoints starts a path of conductances larger than $n^{-\alpha}$ reaching $\partial \B_{m+1}$. Let $e\in\partial \B_{m+1}$ be the last edge of this path. Observe that the conductance of $e$ is larger than $n^{-\alpha}$. There is a path of conductances conductances larger than $n^{-\alpha}$ starting from one of the endpoints of $e$ and reaching $\partial \B_{m+2}$. But since the conductance of $e$ is larger than $n^{-\alpha}$, there is actually a path of conductances conductances larger than $n^{-\alpha}$ starting from any of the endpoints of $e$ and reaching $\partial \B_{m+2}$. Thus we constructed a path from $f$ to $\partial \B_{m+2}$. Iterating this construction, we obtain a path from one endpoint of $f$ to $\partial \B_{n+1}$. 

By Lemma \ref{fontes-mathieu}, all conductances in $B_N$ are greater than $N^{- c }$ for some positive constant $c$ depending on $d$ and $\gamma$. We can choose $n$ large enough such that $N^{- c} \ge n^{-\alpha}$, which ensures the existence of a path of conductances larger than  $n^{-\alpha}$ from one of the endpoints of $e \in \B_n$ to $\partial \B_{n+1}$. 

\medskip
\noindent
$(2)$ For the second assertion of the lemma, we can follow the same reasoning with a slight adaptation. Let $ \gamma > 1/ ( 2 \alpha )$ for some $\alpha \in (0,2)$. Set $\partial B_n= \{x\in B_n : \exists y  \notin B_n\, \text{s.t.}\,  x\sim y\}$, the frontier of $B_n$. 
As before, define $\scrE_n$ to be the event: there exists a vertex $x\in\partial B_{n}$ such that any path from $x$ to the boundary $\partial B_{n+1}$ has at least one edge with conductance less than $n^{-\alpha}$. Then we have by Lemma~\ref{nper} and \eqref{P} that

\begin{equation}
\label{contour'}
\PP \big( {\scrE_n} \big) \le c\, n^{-1- d ( 2 \alpha \gamma - 1)}. 
\end{equation}
The rest of the proof is similar.
\end{proofsect}

\medskip

\begin{proofsect}{Proof of Lemma~\ref{h-inj}} 
First let $\gamma > \frac{1}{8} \frac{d}{d-1/2}$ and choose $\alpha \in (0,2)$ such that $\gamma \alpha ( 4d - 2) > d$. 

Let $m \in \N^\ast$ and set $B_m (z) = (2m+1) z + B_m$ for $z\in\Z^d$. The family $\{B_m (z) \}_{z\in \Z^d}$ constitutes a partition of $\Z^d$.  Note that  $| B_m (z) | \ge m^d$. Then choosing $m= \lfloor (\log n)^{d+1} \rfloor$,  Lemma~\ref{density} and the Borel-Cantelli lemma yield that $\PP$-a.s. for $n$ large enough, the vertices in any $B_m (z)$ with $B_m (z)\cap B_n \neq \emptyset$ belong to $\scrC^\xi$ with a proportion that approaches $1$ when $\xi$ is small enough. We choose $\xi$ small enough such that this proportion is larger than $1/2$. Therefore, for any box $B_m (z)$ that intersects $B_n$, there are sufficiently many sites in $B_m (z) \cap \scrC_n$ to associate with sites in $B_m (z) \cap \scrH_n$ (if any) in an injective way. Let $\varphi$ be an injective map from $\scrH_n$ into $\scrC_n$ such that it associates a site in $B_m (z) \cap \scrH_n$ to a site in $B_m (z) \cap \scrC_n$.

Let us now construct the path $\ell(e, \varphi (e))$ for some edge $e = \{ x, y\}$ and $x \in B_m (z) \cap \scrH_n$ with $B_m (z) \cap B_n \neq \emptyset$. By Lemma~\ref{good path} $(1)$, $\PP$-a.s. for $n$ large enough, for any $e$ of $B_n$, there exists a self-avoiding path, say $(x^1,x^2,x^3,\ldots)$ with $x^1 = x$ or $y$, which reaches from $e$ the boundary of $B_{2n}$ with conductances larger than $(2n)^{-\alpha}$. 
By Lemma \ref{v-h} together with the Borel-Cantelli lemma, $\PP$-a.s. for $n$ large enough, $m$ is larger than the size of a hole. It follows that there is some $k < m$ such that $x^{k} \in \scrC^\xi$. This gives us the first part of the path $\ell(e, \varphi (x))$.

Next we claim that it is possible to join $x^k$ with $\varphi (x) \in B_m (z) \cap\scrC_n$ through a path on $\scrC^\xi$ inside $A_{4m}(x^k) \twopoints = x^k + B_{4m}$ (note that $x^k$ and $\varphi (x)$ belong to $A_{4m}(x^k)$ and that $d(\varphi (x), \partial A_{4m}(x^k))>m$). Indeed, if we suppose that it is not possible to find such a path, there would exist a closed cutset (as seen in Lemma~\ref{nper}) of conductances less than $\xi$ and of diameter at least $m$ separating $x^k$ from $\varphi (x)$ in $A_{4m}(x^k)$. But  Lemma \ref{v-h} rules out this possibility since $m$ is larger than the possible diameter of a hole. 
Therefore, there exists a self-avoiding path from $e$ to $\varphi (x)$ through edges with conductances larger than $(2n)^{-\alpha}$ and of length less than $m+(8m)^d \le ( \log n )^{2 d^2}$.
Note here that this path may leave the box  $B_n$. 

\smallskip
\noindent
$(2)$ The case for which $\gamma > 1/4$ can be treated identically using the assertion $(2)$ of Lemma~\ref{good path}.
\end{proofsect}

\section{\textbf{Spectral gaps estimates}} 
\label{sec:sg}

We work in $L^2 ( \theta )$, the Hilbert space of functions on $\Z^d$ with scalar product 
$$
\langle f, g \rangle = \sum_{x \in \Z^d} f (x) g (x) \theta (x),
$$ 
where $\theta(x)=\pi(x)$ in the CSRW and $\theta(x)=1$ for the VSRW. 

We also define the Dirichlet form 
\begin{equation}
\EE^\omega(f,f)=\dfrac{1}{2}\sum_{\{x,y\} \in\E_d} (f(x)-f(y))^2 \omega_{xy}.
\end{equation}

For both models, CSRW or VSRW, then  $\EE^\omega$ is 
the Dirichlet form on $L^2 ( \theta )$ associated with the corresponding random walk. 

Consider the self-adjoint operator 
\begin{equation}
\label{ggg}
\GG_\omega ( \lambda ) := \LL_\theta -\lambda \MM_\varphi,
\end{equation}
where $\varphi( x ) \deq \1_{\{ x \in\scrC^\xi\}}$ and $\MM_\varphi$ is the multiplicative operator by the function $\varphi$, i.e. $\MM_\varphi f (x) = \varphi (x)  f(x)$. Let $\RR_\omega^t ( \lambda )$ be the semigroup generated by $\GG_\omega ( \lambda )$. The Feynman-Kac formula (see \cite[Proposition 3.3]{B2}) reads 

\begin{equation}
\label{F-K}
\RR_\omega^t ( \lambda ) f(x)=E^x_\omega \left( f(X_t)e^{-\lambda A (t)} \right), \qquad t\geq 0,\, x\in\Z^d.
\end{equation} 

The semigroup of the operator $\GG_\omega ( \lambda )$ with Dirichlet boundary conditions outside the box $B_n$ is given by 
$$
(\RR_{n}^t f)(x) \deq E^x_\omega\left[f(X_t)e^{-\lambda A (t)}; \tau_{B_n} > t \right].
$$
Note that the operator $-\GG_\omega ( \lambda )$ with Dirichlet boundary conditions outside $B_n$ is a non-negative symmetric operator with respect to the restriction of the measure $\theta$ to $B_n$. 
Let $\{\lambda_i, i\in\{1,\ldots, |B_n|\}\}$ be the set of its eigenvalues labelled in increasing order, and $\{\psi_i, i\in\{1,\ldots, |B_n|\}\}$ the corresponding eigenfunctions with due normalization.

Then, by the min-max Theorem and \eqref{ggg}, the eigenvalue $\lambda_1$ is given by

\begin{equation}
\label{V2}
\lambda_1 =  \inf_ {f\not\equiv0}\frac{\EE^{\omega}(f,f)+\lambda\sum_{x\in\scrC_n}f^2(x) \theta (x)}{\sum_{x\in B_n} f^2(x) \theta (x)},
\end{equation}
where the infimum is taken over functions $f$ vanishing outside the box $B_n$.
Recall the notation  $\scrC_n=B_n \cap \scrC^\xi$.

First, we want to prove the following key estimates on $\lambda_1$. 

\begin{lemma}
\label{vph}
$(1)$ Let $X$ be the CSRW and take $\gamma > \frac18 \frac{d}{d-1/2}$. Then there exists $\alpha \in (0,2)$ such that for sufficiently small $\xi$, for a.e. environment, we have for $n$ large enough
\begin{equation}
\label{vphe}
\lambda_1 \ge  n^{- \alpha}
\end{equation}
when we choose $\lambda = ( 1 + 8d / \xi )\, n^{- \alpha}$.

\smallskip
\noindent
$(2)$ For the VSRW, for any $\gamma > 1/4$, there exists $\alpha \in (0,2)$ such that for $\xi$ small enough, for a.e. environment, for $n$ large enough,
\begin{equation}
\label{vphe'}
\lambda_1 \ge  n^{- \alpha}
\end{equation}
when we choose $\lambda = 3 n^{- \alpha}$.
\end{lemma}

To obtain bounds for the exit time as in Proposition~\ref{ll}, we need to estimate another eigenvalue. 

Denote by $\LL_{\scrH_n}$ the generator of the random walk with the
vanishing Dirichlet boundary condition 
on $\scrH_n=B_n \cap \scrH^\xi$. 
The associated semigroup is given by $\p^t_{\scrH_n}=e^{t\mathcal{L}_{\scrH_n}}$.

The operator $- \LL_{\scrH_n}$ is symmetric with respect to the measure $\theta$ and
has $| \scrH_n |$ nonnegative eigenvalues that we enumerate in increasing order and denote as follows:
\begin{equation}
\zeta_1 \le \zeta_2 \le \cdots \le \zeta_{|\scrH_n|}. 
\end{equation}
$\{ \phi_i, i = 1, \ldots, | \scrH_n | \}$ is the set of the associated normalized eigenfunctions. 

The spectral gap $\zeta_1$ admits the variational definition
\begin{equation}
\label{zeta}
\zeta_1 = \inf_{f \neq 0} \frac{\langle - \LL_{\scrH_n} f, f \rangle}{\langle f, f \rangle} = \inf_{f \neq 0} \frac{ \EE^\omega ( f, f )}{\sum_{x \in \scrH_n} f ( x)^2 \theta (x)}, 
\end{equation}
where the infimum is taken over functions $f$ vanishing outside $\scrH_n$.

\begin{lemma}
\label{sg-h}
$(1)$ For the CSRW, for any $\gamma > \frac18 \frac{d}{d-1/2}$, there exists $\alpha \in (0,2)$ such that for sufficiently small $\xi$, for a.e. environment, for $n$ large enough,
\begin{equation}
\label{sg-h'}
\zeta_1 \ge  n^{- \alpha}.
\end{equation}

\smallskip
\noindent
$(2)$ For the VSRW, for any $\gamma > 1/4$, there exists $\alpha \in (0,2)$ such that for $\xi$ small enough, for a.e. environment, for $n$ large enough,
\begin{equation}
\label{sg-hbis}
\zeta_1 \ge  n^{- \alpha}.
\end{equation}
\end{lemma}

\medskip

\begin{proofsect}{Proof of Lemma~\ref{vph}}
$(1)$ Let $\gamma > \frac18 \frac{d}{d-1/2}$; then choose $\alpha' \in (0,2)$ such that $\gamma \alpha' ( 4d - 2) > d$ and $\alpha$ such that $\alpha'<\alpha<2$. 
Let $f$ be a function vanishing outside $B_n$. We use the notation $\text d f(b) \deq f(a) - f(c)$ for any edge $b = \{ a, c\}$.

Let $x \in \scrH_n$ and call $e = \{ x, y \}$ the edge such that $\omega_e = \max_{b \ni x} \omega_b$. 

We use the paths $\ell$ constructed in Lemma \ref{h-inj} to get that 
\begin{equation}
\label{dfe}
f(x) = f (x ) - f (y) + \sum_{b\in \ell( e, \varphi (x))} \text d f(b) + f(\varphi (x)),
\end{equation}  
if the path $\ell( e, \varphi (x))$ starts at $y$. Otherwise, 
\begin{equation}
\label{dfe2}
f(x) = \sum_{b\in\ell( e, \varphi (x))} \text d f(b) + f(\varphi (x)).
\end{equation}
Let us consider the case \eqref{dfe}. 
Observe that Cauchy-Schwarz inequality gives
\begin{equation}
\label{dfer}
 f(x)^2 \leq 2 ( f (x) - f (y) )^2 + 4 \vert \ell(e,\varphi (x))\vert \sum_{b\in\ell( e, \varphi (x))} \text d f(b)^2+ 4 f(\varphi (x))^2.
\end{equation}
Noting that $\pi ( \varphi(x) ) \ge \xi$ and 
\begin{equation}
\label{e-max}
\pi (x) \le 2d\, \omega_e \le 2d,
\end{equation}
we obtain that 
$$
f(x)^2\,\pi(x) \leq 4d ( f (x) - f (y) )^2\omega_e + 8d \vert \ell(e,\varphi (x))\vert \sum_{b\in\ell( e, \varphi (x))} \text d f(b)^2+  \frac{8 d}\xi f(\varphi (x))^2\,\pi(\varphi(x)).
$$

Using the bounds from Lemma \ref{h-inj} (that we apply with $\alpha'$ rather than $\alpha$), we get 
\begin{eqnarray*}
&&f(x)^2\,\pi (x)\\
&\le&
4d ( f (x) - f(y) )^2 \omega_e  + 32 d n^{\alpha'} \, (\log n)^{2 d^2}\,  \,\sum_{b\in\ell( e, \varphi (x))} \text d f(b)^2 \omega_b
+ \frac{8 d}\xi  f(\varphi (x))^2 \pi(\varphi(x)). 
\end{eqnarray*} 
The case \eqref{dfe2} can be treated in the same way
and we have the same inequality. 

Let us now sum this inequality for $x\in\scrH_n$ 
Observe that: - a given edge appears at most $(\log n)^{2 d^3}$ (because of the bound on the length of the path), 
- a given $\varphi(x)$ only appears at most once. 
So 

$$\sum_{x\in \scrH_n} f(x)^2\,\pi (x)
\le 32 d n^{\alpha'} \, (\log n)^{2 d^2}\,  (\log n)^{2 d^3}\, \EE^\omega(f,f)+\frac{8 d}\xi  \sum_{x\in \scrC_n}f(x)^2 \pi(x).
$$

Choose $n$ big enough so that $32 d n^{\alpha'} \, (\log n)^{2 d^2}\,  (\log n)^{2 d^3}\le n^\alpha$. We have obtained the inequality 
\begin{eqnarray*} 
\sum_{x\in \scrH_n} f(x)^2\,\pi (x)\le 
n^{\alpha}\,  \EE^{\omega}( f, f ) + 
\left(\frac{8 d}\xi \right) \sum_{x\in\scrC_n} f^2(x) \pi (x). 
\end{eqnarray*}

To conclude, use the variational formula \eqref{V2}.

\medskip
\noindent
$(2)$ The argument is the same and here we just give an outline of the proof. Let $\gamma > 1/4$, choose $\alpha' \in (0,2)$ such that $\gamma > 1 / (2 \alpha')$ 
and $\alpha$ such that $\alpha'<\alpha<2$. 

Let  $f$ be a function vanishing outside $B_n$. 

Let $x \in \scrH_n$. Then Lemma \ref{h-inj} implies

\begin{equation}
\label{dfe2'}
f(x) = \sum_{b\in\ell( x, \varphi' (x))} \text d f(b) + f(\varphi' (x)),
\end{equation}
which by Cauchy-Schwarz inequality gives
\begin{eqnarray}
\label{dferv}
 f(x)^2 \leq 2 \vert \ell(x,\varphi' (x))\vert \sum_{b\in\ell( x, \varphi' (x))} \text d f(b)^2+ 2 f(\varphi' (x))^2.\nonumber
\end{eqnarray} 
Summing over $\scrH_n$, note that a given edge appears at most $(\log n)^{2 d^3}$ (because of the bound on the length of the path), 
and a given $\varphi(x)$ only appears once. Thus we obtain 
$$
\sum_{x \in \scrH_n} f(x)^2 \leq 2 ( \log n )^{2 d^2}\, ( \log n )^{2 d^3}\, n^{\alpha'} \EE^\omega (f, f)+ 2 \sum_{x \in \scrC_n} f(x)^2,
$$
and hence
$$
\sum_{x\in B_n} f(x)^2 
\leq  
n^{\alpha}\,  \EE^{\omega}( f, f ) + 
3 \sum_{x\in\scrC_n} f^2(x), 
$$
when  $n$ is large enough.
The variational formula \eqref{V2} then yields the desired estimate.
\end{proofsect}

We pass now to the proof of the second spectral gap $\zeta_1$.

\begin{proofsect}{Proof of Lemma~\ref{sg-h}}
The argument is the same as in estimating $\lambda_1$. 

\noindent
$(1)$ Let $X$ be the CSRW and assume that $\gamma > \frac18 \frac{d}{d - 1/2}$.

Suppose $x \in \scrH_{n}$ and call $e = \{ x, y \}$ the edge such that $\omega_e = \max_{b \ni x} \omega_b$. Let $f$ be a function  vanishing outside $\scrH_n$. Then thanks to Lemma~\ref{h-inj}, there exists a path $\ell (e, \varphi (x))$ connecting $e$ to a site $\varphi (x) \in \scrC^\xi \cap B_n$. If this path starts at $y$, write then
\begin{equation}
\label{dfe-z}
f(x) = f (x ) - f (y) + \sum_{b\in \ell( e, \varphi (x))} \text d f(b).
\end{equation}  
Otherwise, write
\begin{equation}
\label{dfe2-z}
f(x) = \sum_{b\in\ell( e, \varphi (x))} \text d f(b).
\end{equation}
Consider the case \eqref{dfe-z} and do the same thing for the second one. 

By Cauchy-Schwarz inequality, \eqref{dfe-z} gives
\begin{equation}
\label{dfer-z}
 f(x)^2 \leq 2 ( f (x) - f (y) )^2 + 2 \vert \ell(x,\varphi (x))\vert \sum_{b\in\ell( e, \varphi (x))} \text d f(b)^2.
\end{equation}
Multiply \eqref{dfer-z} by $\pi (x)$ and use \eqref{e-max} and \eqref{ell} to obtain
\begin{equation}
\label{dfer2}
 \sum_{x \in \scrH_n} f(x)^2 \pi (x) \leq 4 d \sum_{x \in \scrH_n} ( f (x) - f (y) )^2 \omega_e + 8 d\, ( \log n )^c\, n^{\alpha'} \,
 \EE^\omega ( f, f),
\end{equation}
where $\alpha' \in (0, 2)$ is chosen such that $\gamma \alpha' ( 4d - 2) > d$ and we used again the fact that a given edge appears at most $(\log n)^{2 d^3}$ (because of the bound on the length of the path).  Thus for $\alpha \in ( \alpha', 2)$ and $n$ large enough,
$$
\text{R.H.S. of \ref{dfer2}} \le n^\alpha \EE^\omega (f, f ).
$$
which, using \eqref{zeta}, gives the lower bound \eqref{sg-h}.

\medskip
\noindent
$(2)$
As for the VSRW, instead of \eqref{dfe-z}, we have by Lemma~\ref{h-inj},
\begin{equation*} 
f(x) = \sum_{b\in \ell( x, \varphi (x))} \text d f(b).
\end{equation*} 
The remainder of the proof is the same. 
\end{proofsect}

\section{\textbf{Proof of Proposition~\ref{ll}}}\label{twc}

With all the necessary tools in hand, we can finally provide the proof of Proposition~\ref{ll}.

\begin{proofsect}{Proof of Proposition~\ref{ll}}
$(1)$ Let $X$ be the CSRW. Take $\gamma > \frac{1}{8} \frac{d}{d - 1/2}$ and let $\alpha \in (0,2)$ be as in Lemma~\ref{vph}. Choose $\delta > 0$ such that 
\begin{equation}
\label{galpha}
1 - \alpha\, \frac{(1+\delta)}{2}  >  0.
\end{equation}  

Let $n =  t^{(1+\delta)/2}$ and suppose $x\in B_{n/2}$. 
Observe that for any constant $\lambda>0$ (may be $\PP$-random) and any $\varepsilon\in(0,1)$, Chebyshev's  inequality gives 
\begin{eqnarray}
\label{i11}
P^x_\omega (A (t)\leq  \varepsilon\, t)
&=&
P^x_\omega(A (t)\leq  \varepsilon\, t; \tau_{B_n} > t)+ P^x_\omega(A (t)\leq \varepsilon\, t; \tau_{B_n} \leq t)\nonumber
\\
&\leq&
P^x_\omega\left(e^{-\lambda A (t)}\geq e^{-\varepsilon \lambda t}; \tau_{B_n}  > t \right) + P^x_\omega(\tau_{B_n} \leq t)\nonumber
\\
&\leq&\label{TC}
e^{\varepsilon \lambda t} \,E^x_\omega\left(e^{-\lambda A (t)}; \tau_{B_n}  > t\right)+ P^x_\omega( \tau_{B_n} \leq t).
\end{eqnarray}
By \cite[Proposition 4.7]{ABDH}), we have for $t$ large enough,
\begin{equation}
\label{c-v}
P^x_\omega(\tau_{B_n}  \leq t )\leq C\, e^{- c t^\delta},
\end{equation}
where $C, c$ are numerical constants. 

Let us look now at the first term of the right hand side of \eqref{TC}. Recall the eigenvalues $\{\lambda_i, i\in\{1,\ldots, |B_n|\}\}$ of 
the restricted operator $- \GG_\omega ( \lambda)$ 
and their associated normalized eigenfunctions $\{\psi_i, i\in\{1,\ldots, |B_n|\}\}$. For $f = \1_{B_n}$, observe first that

\begin{eqnarray}
\label{sg-decomp}
(\RR^t_{n} f) (x)
&=&
E^x_\omega\left(e^{-\lambda A (t)};  \tau_{B_n} > t \right)
= \sum_{i=1}^{|B_n|} e^{-\lambda_i t}\left\langle f, \psi_i\right\rangle\psi_i(x).
\end{eqnarray}
Then
 \begin{equation*}
(\RR^t_{n} f)^2(x) \pi (x)
\le
\sum_{y\in B_n} (\RR^t_{n} f)^2(y) \pi (y)
=
\sum_i e^{-2\lambda_i t}\left\langle f, \psi_i\right\rangle^2
\le e^{-2\lambda_1 t} \Vert f \Vert^2_2, 
\end{equation*}
which is less than
 \begin{equation*}
2d\, |B_n|\, e^{-2\lambda_1 t}.
\end{equation*}
Thus  by Lemma~\ref{vph} (choosing $\lambda=cn^{-\alpha}$) 
and using the fact that $1/\pi (x)\leq n^{c}$, $c>0$ being a constant that depends only on $d$ and $\gamma$ (cf. Lemma \ref{fontes-mathieu}), we obtain
\begin{equation}
\label{QSE}
e^{\varepsilon \lambda t}\, {E^x_\omega} \big(e^{-\lambda A (t)}; \tau_{B_n}  > t \big)
\le 
C\, n^{d + c}\, e^{-  ( 1 - \varepsilon ) t^{1 - \alpha ( 1+\delta)/2}}.
\end{equation}
According to \eqref{galpha} and since $\varepsilon \in (0, 1)$, for large enough $t$, we have  
\begin{equation}
\text{R.H.S. of \eqref{QSE}}
\le
c\, e^{- (1 - \varepsilon) t^\sigma}\label{qle} 
\end{equation}
for any $\sigma < 1 - \alpha (1 + \delta)/2$. 
Thus \treeeqref{TC}{c-v}{qle} give the desired upper bound for any $\sigma$ small enough.

\medskip

As for the exit time estimate, suppose $x \in \scrH_n = B_n \cap\scrH^\xi$ with $n = t^{(1+\delta)/2}$. Recall the eigenvalues $\{\zeta_i, i\in\{1,\ldots, |\scrH_n|\}\}$ of the restricted operator $- \LL_{\scrH_n}$ and their associated normalized eigenfunctions $\{\phi_i, i = 1,\ldots, |\scrH_n|\}$. Let $f = \1_{\scrH_n}$ and observe that
\begin{equation}
P^x_\omega ( \tau_h > t/2 ) = \p^{t/2}_{\scrH_n}  f (x) = \sum_{i=1}^{|{\scrH_n}|} e^{-\zeta_i t/2}\left\langle f, \phi_i\right\rangle \phi_i(x)
\end{equation}
which, by Lemmas~\tworef{sg-h}{v-h}, yields that
\begin{equation}
P^x_\omega ( \tau_h > t/2 )
\le 
\frac{e^{- \zeta_1 t/2}}{\sqrt{\pi ( x )}} \,   \Vert f \Vert_2 
\le
\frac{|B_n|}{\sqrt{\pi ( x )}}\, e^{- \zeta_1 t/2}
\le 
n^{d+c/2}\, 
e^{- \frac12 t^{1 - \alpha (1+ \delta)/2} }
\end{equation}
where we used again that $1/\pi (x)\leq n^{c}$. The claim follows for any $\sigma <  1 - \alpha (1+ \delta)/2$.

\smallskip
\noindent
$(2)$
Clearly, the above argument for the CSRW holds for the VSRW with $\gamma > 1/4$ and the counting measure instead of $\pi$.
\end{proofsect}

\section{Appendix}\label{appen}

Here, we give some relatively 
standard proofs for completeness. 

\begin{proofsect}{Proof of Proposition \ref{Pr:11.1}} 
Fix a large ball $B(x_0,R)$, and denote for simplicity $B_r = B(x_0,r)$. Let us prove that, for any
$R_1(x_0)\le r< R/6$,
\begin{equation}
\label{11.1-bis}
\osc_{B_r}\, u \le (1-\delta) \osc_{B_{3r}}\, u, 
\end{equation}
where $\delta = \delta(C_E) \in (0,1)$. Then \eqref{Pr:11.1-eq} follows from \eqref{11.1-bis} by iterating.

The function $u - \min_{B_{3r}} u$ is nonnegative in $B_{2r}$ and
harmonic in $B_{2r}$ . Applying  Assumption \ref{asmp:1-1asp} (iii) to this function, we obtain
$$\max_{B_r} u - \min_{B_{3r}} u \le C_E (\min_{B_r} u - \min_{B_{3r}} u),$$
for all $R_1(x_0)\le r\le R/6$, so 
$$\osc_{B_r} u \le (C_E - 1) (\min_{B_r} u - \min_{B_{3r}} u).$$
Similarly, we have
$\osc_{B_r} u \le (C_E - 1) (\max_{B_{3r}} u - \max_{B_{r}} u).$
Summing up these two inequalities, we get
$$(1+C_E) \osc_{B_r} u \le (C_E-1)\, \osc_{B_{3r}} u,$$
whence \eqref{11.1-bis} follows.
\end{proofsect}

\begin{proofsect}{Proof of Proposition \ref{Pr:11.2}}  
Denote for simplicity $B_r = B(x_0,r)$. Let 
\[
g_{B_R}(x,y)=\int_0^\infty p_t^{B_R}(x,y)dt. 
\]
We then have
$$u(y) = - \sum_{z\in B_R} g_{B_R} (y,z) f(z)\theta_z,$$
and since 
$E (x_0,R) = \sum_{y\in B_R} g_{B_R} (x_0,y)\theta_y$, 
we obtain
$$\max_{B_R} |u| \le \overline E (x_0,R) \,\max_{B_R} |f|.$$
Let $v$ be a function on $B_r$ that solves the Poisson equation $\LL v = f$ in $B_r$. In the same way
$$\max_{B_r} |v| \le \overline E (x,r) \,\max_{B_r} |f|.$$
The function $w = u - v$ is harmonic in $B_r\subset B_R$ whence, by Proposition~\ref{Pr:11.1},
$$\osc_{B_{\sigma r}}\, w \le \varepsilon \osc_{B_r}\, w\qquad \forall (\sigma r)\ge R_1(x_0).$$
Since $w = u$ on $B_R \setminus B_r$, the maximum principle implies that
$$\osc_{B_r} w\le \osc_{B_R} w = \osc_{\overline B_R\setminus B_r} w = \osc_{\overline B_R\setminus B_r} u \le 2\, \max_{B_R} |u|.$$
Hence,
$$\osc_{B_{\sigma r}} u \le \osc_{B_{\sigma r}} v + \osc_{B_{\sigma r}} w \le 2\, \max_{B_{\sigma r}} |v| + 2 \varepsilon\, \max_{B_R} |u| \le 2(\overline E(x_0, r) + \varepsilon\, \overline E(x_0, R)) \max_{B_R} | f |,$$
\end{proofsect}

\begin{proofsect}{Proof of Proposition \ref{PR:12.1}}   
(i) Let $\LL^ A_{V}$ be the restriction of the operator $\LL_V$ on $A$ with Dirichlet boundary conditions outside $A$ and denote by $\{\lambda_i, i\in\{1,\ldots, |A|\}\}$ be the set of eigenvalues of the positive symmetric operator $-\LL^ A_{V}$ labelled in increasing order, and $\{\psi_i, i\in\{1,\ldots, |A|\}\}$ the corresponding eigenfunctions with due normalization. We have
$$u_t = P^A_t f = \sum_i e^{-\lambda_i t} \left\langle f, \psi_i\right\rangle \psi_i,$$
which gives 
$$- \partial_t\, u_t = \sum_i \lambda_i e^{-\lambda_i t} \left\langle f, \psi_i\right\rangle \psi_i,$$
and thus
$$\| \partial_t\, u_t \|^2_2 = \sum_i \lambda^2_i e^{- 2\lambda_i t} \left\langle f, \psi_i\right\rangle^2.$$

Using the inequality $\lambda_i s \le  e^{\lambda_i s}$, we get
$$\| \partial_t\, u_t \|^2_2 \le \frac{1}{s^2}  \sum_i e^{- 2 \lambda_i (t-s)}  \left\langle f, \psi_i\right\rangle^2 = \frac{1}{s^2}\, \| u_{t-s} \|^2_2.$$
(ii) We have the semigroup identity
$$p^ A_t(x,y) = \sum_{z} p^ A_v(x,z) p^ A_{t-v} (z,y)\theta_z,$$
from which we get
$$\partial_t\, p^A_t(x,y) = \sum_z p^ A_v(x,z) \partial_t\,\, p^ A_{t-v} (z,y)\theta_z,$$ 
whence
$$\left|\partial_t\, p^A_t(x,y)\right| \le \|  p^ A_v(x,\cdot) \|_2 \|\partial_t\, p^ A_{t-v} (y,\cdot)\|_2,$$
By Proposition~\ref{PR:12.1}\,(i), 
$$\|\partial_t\, p^ A_{t-v} (y,\cdot)\|_2 \le \frac1s \|\partial_t\, p^ A_{t-v-s} (y,\cdot)\|_2$$
for any $s\le t-v$. Since
$$\|p^A_v (x,\cdot)\|^2_2 = \sum_z p^A_v (x,z)^2\theta_z = p^A_{2v} (x,x),$$
we obtain \eqref{Pr:12.2-eq}.

\noindent 
(iii) Choose $v\simeq s\simeq t/3$, it follows then from Assumption \ref{asmp:1-1asp} (i) that for any nonempty finite set $A \subset \Z^d$ and $t$ large enough,
$$p^A_{2v} (x,x) \le C\, t^{-d/2} ~~\text{and }~~ p^A_{2(t-v-s)} (y,y) \le C\, t^{-d/2},
~\,~~~\forall x,y\in B(x_0, t^{(1+\delta)/2}),$$
when $2t/3\ge T_0(x_0)$, 
whence by Proposition~\ref{PR:12.1}\,(ii),
$$\left| \partial_t\, p^A_t (x,y) \right| \le C \,t^{-(\frac d2+1)}.$$
By letting $A\rightarrow \Z^d$, we obtain \eqref{Pr:12.3-eq}.
\end{proofsect}

\begin{proofsect}{Proof of Theorem \ref{thm:PHI}}
Given Proposition \ref{Pr:13.1}, we can use the balayage argument as in the proof of  
\cite[Theorem 3.1]{BH-LCLT}. Note that the statement of \cite[Theorem 3.1]{BH-LCLT} includes
estimates of very good balls, but as in the proof, we only need the heat kernel estimates.

Let $C'>0$ be slightly larger than $1$, $C_*=\delta_0^{-1}C'$ 
and define 
\begin{eqnarray*}
&B=B(x_0,C_*R),~~ 
B_1=B(x_0, C'R), \\
&Q=Q(x_0,R,R^2)=(0,4R^2]\times B, E=(0,4R^2]\times B_1.
\end{eqnarray*}
Let $u(t,x)\ge 0$ be caloric on $Q$.
Let $Z$ be the space-time process on $\R \times G$
given by $Z_t=(V_0-t, X_t)$, where $X$ is the Markov chain on $G$, and $V_0$ is the 
initial time. 
Define $u_E$ by
\[
 u_E(t,x)= E^x \big( u(t-T_E, X_{T_E}); T_E< \tau_Q \big),   
\]
where $T_E=\inf\{t\ge 0: Z_t\in E\}$ and $\tau_Q=\inf\{t\ge 0: Z_t\notin Q\}$. 
Clearly, $u_E=u$ on $E$, $u_E=0$ on $Q^c$, and $u_E\le u$ on $Q-E$. 
Since a dual process of $Z$ exists and can be written as $(V_0+t, X_t)$, the balayage formula holds
and we can write
\[
u_E(t,x)=  \int_E p^B_{t-r}(x,y) \nu_E(dr,dy), \quad (t,x) \in Q, 
\]
for a suitable measure $\nu_E$. 
Here $p^B_t(x,y)$ is the heat kernel of $X$, killed on exiting from $B$. 
In this case we can write things more
explicitly. Set
\begin{equation}\label{J-def}
 J f(x) = 
\begin{cases}
 \sum_{y \in B} \frac{\omega_{xy}}{\theta(y)} f(y) 
 &\hbox{ if } x \in B_1, \\
 0  & \hbox{ if } x \in B-B_1. 
\end{cases}
\end{equation}
The balayage formula takes the form 
\begin{equation}\label{bal-cts}
 u_E(t,x) = \sum_{y \in B_1} p^B_t(x,y) u(0,y)\theta(y)
+   \sum_{y \in B_1}\int_{(0,T]} p^B_{t-s}(x,y)k(s,y) \theta(y) ds,
\end{equation}
where $k(s,y)$ is zero if $y \in B-B_1$ and  
\begin{equation}\label{k-cts}
   k(s,y) = J (u(s,\cdot)-u_E(s,\cdot))(y), \quad y \in B_1.
\end{equation}
(See \cite[Proposition 3.3]{BBK2}; See also \cite[Appendix]{BH-LCLT} for a self-contained proof of \eqref{bal-cts}
and \eqref{k-cts} for the discrete time case.)
Since $u=u_E$ on $E$, if $s>0$ then \eqref{k-cts} implies
that $k(r,y)=0$ unless $y \in \partial (B-B_1)$.

Now let $(t_1,y_1) \in Q_-$ and $(t_2,y_2) \in Q_+$.
Note that since $(t_i,y_i)\in E$ for $i=1,2$, 
we have $u_E(t_i,y_i)=u(t_i,y_i)$. 
Choose $R_5(x_0)$ large enough such that $R_5(x_0)\ge C(R_*(x_0)+\sqrt {T_0(x_0)}+\sqrt {T_1(x_0)})$
for some $C\ge 1$. 
By Assumption \ref{asmp:1-1asp} (i), Proposition \ref{Pr:13.1} and 
Corollary \ref{thm:corhke}, 
 we have, writing
$A=  \partial (B-B_1)$ and $T=R^2$, 
\begin{align*}
 p^B_{t_2-s}(x,y) & \ge  c_1 T^{-d/2}
  \quad \hbox{ for } x, y \in B_1, \, 0\le s \le T, \\
  p_{s}(x,y) &\le  c_2  T^{-d/2}
  \quad \hbox{ for } x, y \in B_1, \, T \le s \le 2T, \\
   p_{t_1-s}(x,y) &\le c_2 T^{-d/2}
  \quad \hbox{ for } x \in B, 
  \, y \in A, \, 0 < s \le t_1. 
\end{align*}
Substituting these bounds in \eqref{bal-cts}, we have 

\begin{align*}
   u(t_2,y_2) &=
 \sum_{y \in B_1} p^B_{t_2}(y_2,y) u(0,y)\theta (y)
+  \sum_{y \in A} \int_0^{t_2}  p^B_{t_2-s}(y_2,y)k(s,y) \theta (y) ds \\
 &\ge 
 \sum_{y \in B_1} c_1 T^{-d/2} u(0,y)\theta (y)
+  \sum_{y \in A} \int_{0}^{t_1} c_1 T^{-d/2} k(s,y) \theta (y)ds  \\
 &\ge 
 \sum_{y \in B_1} c_1 c_2^{-1} p^B_{t_1}(y_1,y) u(0,y)\theta (y)
+  \sum_{y \in A} \int_0^{t_1}c_1 c_2^{-1}  p^B_{t_1-s}(y_1,y)k(s,y) 
 \theta (y)ds \\
 &=  c_1 c_2^{-1} u(t_1,y_1),
\end{align*}
 which proves \eqref{eq:PHI-u}. 
\end{proofsect} 

\section*{Dedication}
\noindent 
Omar Boukhadra wishes to dedicate this paper to the memory of his father Youcef Bey.

\ \newline

\section*{Acknowledgments}
\noindent 
The authors thank an 
anonymous reviewer for the suggestions and comments which helped to improve the paper. 
This research was supported by the french ANR project MEMEMO2 2010 BLAN 0125,
and by the Grant-in-Aid for Scientific Research (A) 25247007, Japan.

\end{document}